\documentclass[11pt]{amsart}
\usepackage{amscd}                                            
\usepackage[arrow,matrix]{xy}
\usepackage{graphicx}
\usepackage{hyperref}
\usepackage{comment}
\usepackage{color}

\newcommand{\green}[1]{\textcolor{green}{#1}}
\newcommand{\blue}[1]{\textcolor{blue}{#1}}
\newcommand{\brown}[1]{\textcolor{brown}{#1}} 
\definecolor{brown}{RGB}{150,100,0}
\usepackage{amsmath, latexsym, amssymb}
\usepackage{soul}
\input xypic
\numberwithin{equation}{section}
\theoremstyle{plain}
\newtheorem{lemma}{Lemma}[section]
\newtheorem{proposition}[lemma]{Proposition}
\newtheorem{theorem}[lemma]{Theorem}

\theoremstyle{definition}
\newtheorem{definition}[lemma]{Definition}
\newtheorem{remark}[lemma]{Remark}
\newtheorem{example}[lemma]{Example}

\begin{document}
\newcommand{\R}{{\mathbb R}}
\newcommand{\C}{{\mathbb C}}
\newcommand{\E}{{\mathbb E}}
\newcommand{\F}{{\mathbb F}}
\renewcommand{\O}{{\mathbb O}}
\newcommand{\Z}{{\mathbb Z}} 
\newcommand{\N}{{\mathbb N}}
\newcommand{\Q}{{\mathbb Q}}
\renewcommand{\H}{{\mathbb H}}

\newcommand{\Aa}{{\mathcal A}}
\newcommand{\Bb}{{\mathcal B}}
\newcommand{\Cc}{{\mathcal C}}    
\newcommand{\Dd}{{\mathcal D}}
\newcommand{\Ee}{{\mathcal E}}
\newcommand{\Ff}{{\mathcal F}}
\newcommand{\Gg}{{\mathcal G}}    
\newcommand{\Hh}{{\mathcal H}}
\newcommand{\Kk}{{\mathcal K}}
\newcommand{\Ii}{{\mathcal I}}
\newcommand{\Jj}{{\mathcal J}}
\newcommand{\Ll}{{\mathcal L}}    
\newcommand{\Mm}{{\mathcal M}}    
\newcommand{\Nn}{{\mathcal N}}
\newcommand{\Oo}{{\mathcal O}}
\newcommand{\Pp}{{\mathcal P}}
\newcommand{\Qq}{{\mathcal Q}}
\newcommand{\Rr}{{\mathcal R}}
\newcommand{\Ss}{{\mathcal S}}
\newcommand{\Tt}{{\mathcal T}}
\newcommand{\Uu}{{\mathcal U}}
\newcommand{\Vv}{{\mathcal V}}
\newcommand{\Ww}{{\mathcal W}}
\newcommand{\Xx}{{\mathcal X}}
\newcommand{\Yy}{{\mathcal Y}}
\newcommand{\Zz}{{\mathcal Z}}

\newcommand{\zt}{{\tilde z}}
\newcommand{\xt}{{\tilde x}}
\newcommand{\Ht}{\widetilde{H}}
\newcommand{\ut}{{\tilde u}}
\newcommand{\Mt}{{\widetilde M}}
\newcommand{\Llt}{{\widetilde{\mathcal L}}}
\newcommand{\yt}{{\tilde y}}
\newcommand{\vt}{{\tilde v}}
\newcommand{\Ppt}{{\widetilde{\mathcal P}}}
\newcommand{\bp }{{\bar \partial}}

\newcommand{\ad}{{\rm ad}}
\newcommand{\Om}{{\Omega}}
\newcommand{\om}{{\omega}}
\newcommand{\eps}{{\varepsilon}}
\newcommand{\Di}{{\rm Diff}}
\newcommand{\vol}{{\rm vol}}
\newcommand{\Pro}[1]{\noindent {\bf Proposition #1}}
\newcommand{\Thm}[1]{\noindent {\bf Theorem #1}}
\newcommand{\Lem}[1]{\noindent {\bf Lemma #1 }}
\newcommand{\An}[1]{\noindent {\bf Anmerkung #1}}
\newcommand{\Kor}[1]{\noindent {\bf Korollar #1}}
\newcommand{\Satz}[1]{\noindent {\bf Satz #1}}

\renewcommand{\a}{{\mathfrak a}}
\renewcommand{\b}{{\mathfrak b}}
\newcommand{\e}{{\mathfrak e}}
\renewcommand{\k}{{\mathfrak k}}
\newcommand{\pg}{{\mathfrak p}}
\newcommand{\g}{{\mathfrak g}}
\newcommand{\gl}{{\mathfrak gl}}
\newcommand{\h}{{\mathfrak h}}
\renewcommand{\l}{{\mathfrak l}}
\newcommand{\sm}{{\mathfrak m}}
\newcommand{\n}{{\mathfrak n}}
\newcommand{\s}{{\mathfrak s}}
\renewcommand{\o}{{\mathfrak o}}
\renewcommand{\so}{{\mathfrak so}}
\renewcommand{\u}{{\mathfrak u}}
\newcommand{\su}{{\mathfrak su}}
\newcommand{\ssl}{{\mathfrak sl}}
\newcommand{\ssp}{{\mathfrak sp}}
\renewcommand{\t}{{\mathfrak t }}
\newcommand{\X}{{\mathfrak X}}

\newcommand{\pb}{{\mathbf p}}
\newcommand{\rk}{{\rm rk}}
\newcommand{\grad}{{\nabla}}
\newcommand{\sgn}{{\text{ sgn}}}

\newcommand{\Cinf}{C^{\infty}}
\newcommand{\la}{\langle}
\newcommand{\ra}{\rangle}
\newcommand{\half}{\scriptstyle\frac{1}{2}}
\newcommand{\p}{{\partial}}
\newcommand{\notsub}{\not\subset}
\newcommand{\iI}{{I}}               
\newcommand{\bI}{{\partial I}}      
\newcommand{\LRA}{\Longrightarrow}
\newcommand{\LLA}{\Longleftarrow}
\newcommand{\lra}{\longrightarrow}
\newcommand{\LLR}{\Longleftrightarrow}
\newcommand{\lla}{\longleftarrow}
\newcommand{\INTO}{\hookrightarrow}

\newcommand{\QED}{\hfill$\Box$\medskip}
\newcommand{\UuU}{\Upsilon _{\delta}(H_0) \times \Uu _{\delta} (J_0)}
\newcommand{\bm}{\boldmath}

\newcommand{\commentgreen}[1]{\green{(*)}\marginpar{\fbox{\parbox[l]{3cm}{\green{#1}}}}}
\newcommand{\commentred}[1]{\textcolor{red}{(*)}\marginpar{\fbox{\parbox[l]{3cm}{\textcolor{red}{#1}}}}}
\newcommand{\commentblue}[1]{\blue{(*)}\marginpar{\fbox{\parbox[l]{3cm}{\blue{#1}}}}}
\newcommand{\commentbrown}[1]{\brown{(*)}\marginpar{\fbox{\parbox[l]{3cm}{\brown{#1}}}}}

\newcommand{\er}{\mathbb{R}}
\newcommand{\ce}{\mathbb{C}}
\newcommand{\be}{\begin{equation}}
\newcommand{\bel}[1]{\begin{equation}\label{#1}}
\newcommand{\qe}{\end{equation}}
\newcommand{\ee}{\end{equation}}
\newcommand{\eeq}{\end{equation}}
\newcommand{\ba}{\begin{eqnarray}}
\newcommand{\ea}{\end{eqnarray}}
\newcommand{\rf}[1]{(\ref{#1})}

\newcommand{\T}{{\mathbf T}}

\title[Singular   statistical models]{The Cram\'er-Rao   inequality   on singular  statistical models  I}

\author{J\"urgen Jost}
\address{Max-Planck-Institute for Mathematics in the Sciences, Leipzig, Germany}
\email{jjost@mis.mpg.de}

\author{H\^ong V\^an L\^e}
\address{Institute of Mathematics CAS,
Zitna 25, 11567 Praha 1, Czech Republic}
\email{hvle@math.cas.cz}

\author{Lorenz Schwachh\"ofer}
\address{Fakult\"at f\"ur Mathematik,
Technische Universit\"at Dortmund,
Vogelpothsweg 87, 44221 Dortmund, Germany}
\email{Lorenz.Schwachhoefer@math.uni-dortmund.de} 

 \date{\today}
 \thanks {H.V.L. is partially supported by RVO: 67985840}

\medskip

\abstract   We  introduce   the notion  of  the essential tangent  bundle of a   parametrized  measure model   and the   notion of reduced   Fisher metric on  a    
(possibly singular)  2-integrable  measure  model.
Using these notions and  a new  characterization  of $k$-integrable parametrized  measure models, we     extend  the  Cram\'er-Rao   inequality  to  $2$-integrable    (possibly singular) statistical models      for  general  $\varphi$-estimations, where $\varphi$  is  a  $V$-valued   feature   function  and $V$ is a topological  vector  space. Thus  we derive    an intrinsic  Cram\'er-Rao inequality   in the most  general terms  of parametric statistics. 
\endabstract

\maketitle 
 
\tableofcontents

\section{Introduction}\label{sec:intro}

A statistical model, also called a learning machine, is a basic notion in mathematical statistics, statistical learning theory and their  applications \cite{Amari1987, Amari2016, AN2000, Vapnik1999,  Watanabe2009}. 

The basic task is to infer the parameter of the model from observations of samples of the underlying distribution. For that purpose, the map taking the parameters of the model to
probability distributions needs to be one-to-one. Furthermore, for applying the Cram\'er-Rao inequality, the Fisher information matrix of the model should be  positive definite. These are considered as limitations of a statistical model, and a model not satisfying these requirements is called singular (the precise terminology varies somewhat in the literature, see e.g. \cite[Definition 1.7, p. 10]{Watanabe2009}, \cite[p. 12]{BKRW1998}). Such singular statistical models  
appear  in statistics  ubiquitously, however,   and we cannot ignore   singular points  for  estimation  problems
\cite{Watanabe2007, Watanabe2009}. Here, we deal with such possibly singular models. The simple starting observation is that, even if the model parameter cannot be fully inferred, the observations in general will still restrict its possibilities, and even if the Fisher information matrix is degenerate in some directions, there will be others in which it is positive definite, and these can still be used to control some of the variance. In this contribution, we shall set up a systematic mathematical framework to handle that issue, that is, derive estimates that cover cases where the Fisher information matrix is not strictly positive definite.

In this  paper  we  call a point  $\xi \in M$ {\it singular}  if  the Fisher information matrix  at $\xi$ is degenerate,
and we     call      a point $\xi \in M$  {\it unidentifiable}, if  $\# (\pb ^{-1} (\pb (\xi))) \ge 2$, following the  terminology of  Amari and Watanabe.   We should also point out that, in contrast to   \cite{BKRW1998}, we do not  require   that  a regular point  must be an interior
point.  In  particular,   our  statistical models   include  Banach manifolds with  boundary, where  the  boundary   can have singularities, i.e.  the Fisher  metric can be degenerate at  boundary points.
To be general,  as in \cite{Watanabe2009}, we say {\it singular} when we really mean {\it possibly singular}, that is, we always implicitly include  regular 
statistical  models, in particular, when we don't  specify the singular points.

Thus, here we shall deal with such singular statistical models, and our main achievement will be a corresponding version of the Cram\'er-Rao   inequality. 
 Until the  present  paper, see also \cite{AJLS2016} for closely related results, the Cram\'er-Rao's   inequality in the context of parametric statistics   was known to   hold only on     statistical models
where the Fisher information matrix is positive definite, see  Subsection \ref{subs:conclusion}  for  a more  detailed comment.   For that reason,  the singular  statistical models  considered in \cite{Watanabe2007, Watanabe2009}   are  supposed to be  real  analytic varieties so that    parameter   estimation  problems  can  be simplified. In our   paper 
we     introduce  the  notions of   essential tangent  bundle, reduced  Fisher metric,  visible functions and their generalized  gradient and pre-gradient.   Using these new notions  and a new characterization 
of $k$-integrable parametrized measure  models, 
we  extend the   Cram\'er-Rao  inequality    on   singular   statistical  models  for  general $\varphi$-regular estimations, where
$\varphi$ is a $V$-valued feature  and $V$ is a topological vector  space.
The  most   complete treatment currently available of
statistical  models in  the context of parametric statistics   has     been developed in  \cite{AJLS2015}, \cite{AJLS2016b}, \cite{AJLS2016}, see    
Subsection \ref{subs:conclusion} for  more detailed comments.  Our  new treatment of the Cram\'er-Rao  inequality  is  therefore a natural consequence of this new theory of parametric statistics.

  In the subsequent paper \cite{JLS2017b} we      study conditions   for the existence of efficient estimators on singular statistical models.  In particular  we    prove the existence of  (possibly  biased)
efficient estimators  on a class  of strictly singular  finite dimensional  statistical models  and the existence  of biased efficient estimators  on a large class of Fukumizu's   infinite  dimensional   exponential  manifolds.  

We are working within the context of parametric statistics.  Note that,  unlike the     accepted convention in \cite{BKRW1998}, \cite{Wasserman2006}, our  parametric   statistical models  may be  infinite dimensional.
Since nonparametric statistics is conceptually and methodologically different, naturally also the Cram\'er-Rao inequality takes a somewhat different form there. Nevertheless, also the nonparametric approach as in  \cite{Janssen2003} can deal with possibly singular situations. In that context, the most advanced version of the Cram\'er-Rao inequality seems to be that of Janssen \cite{Janssen2003}. At the end of this paper, we shall compare his version with ours, to the extent that a comparison between a parametric and a nonparametric approach is feasible.

Our paper is organized as follows.

In   Section \ref{sec:reduced}  we recall the  notion  of  parametrized  measure (resp. statistical)
model that has been introduced in \cite{AJLS2015} and refined  in \cite{AJLS2016b}.   We  prove  a new  characterization of $k$-integrable
parametrized measure  models. Then
we  introduce the notion of the essential tangent space    and reduced Fisher metric  which are crucial for the   extension of the Cram\'er-Rao inequality to  singular statistical models.

In Section \ref{sec:vis}  we  introduce  a  large class of visible  functions
on statistical models that encompass    estimators     considered  in our  general
Cram\'er-Rao  inequality. 
 We also introduce
the notion of   the generalized gradient   and a pre-gradient  of  a visible function. 
We prove  that
differentiation  under the integral sign is valid for  regular   visible functions  associated to  estimators. 
This  is a   technical  important point  in the proof  of   the Cram\'er-Rao  inequality in  Section \ref{sec:CR}.   At the end  of  Section \ref{sec:vis}  we illustrate  our theory in the  case   of finite  sample  spaces,  whose results shall be used  in the  second part  of this paper \cite{JLS2017b}.

In Section \ref{sec:CR} we      prove   a general  Cram\'er-Rao  inequality  
and  derive from it classical  Cram\'er-Rao inequalities. We also compare  
	our results with  other  generalizations  of the   Cram\'er-Rao inequalities.
	Finally we summarize our main contributions  at the end of the paper.

\

{\bf  Notations}.  For a   measurable  space $\Om$  and a finite measure  $\mu_0$ on $\Om$  we denote  \begin{eqnarray*}
\Pp(\Om) & := & \{ \mu \;:\; \mu\; \mbox{a probability measure on $\Om$}\}\\
\Mm(\Om) & := & \{ \mu \;:\; \mu\; \mbox{a finite measure on $\Om$}\}\\
\nonumber
\Ss(\Om) & := & \{ \mu \;:\; \mu\; \mbox{a signed finite measure on $\Om$}\}\\
\nonumber
  \Ss(\Om, \mu_0) & = & \{\mu = \phi \, \mu_0 \; : \; \phi \in L^1(\Om, \mu_0) \}.
\end{eqnarray*}

Then $\Ss(\Om)$ is a Banach space whose norm $\| \cdot \|_{TV}$ is given by the total variation, and $\Ss(\Om, \mu_0) \subset \Ss(\Om)$ is a closed subspace whose norm is given by $\|\phi \mu_0\|_{TV} = \|\phi\|_1$, where the latter refers to the the norm in $L^1(\Om, \mu_0)$.

\section{$k$-integrable    parametrized  measure  models  and reduced  Fisher  metric}\label{sec:reduced}

In this section  we  recall  the notion  of a $k$-integrable parametrized measure model (Definition \ref{def:param-measure-model}) that has been introduced in \cite{AJLS2016b}.  The  concept of $2$-integrability (resp. $3$-integrability)  is required  for  the right concept  of  the    Fisher  metric (resp. the Amari-Chentsov tensor)   on parametrized measure models, see \cite{AJLS2015}. 
 We prove the existence  of   a   dominating measure  under a mild condition (Proposition \ref{prop:mfld-dom}), which is important  for our proof  of the  general Cram\'er-Rao inequality in    later sections.
Then   we give a characterization of $k$-integrability (Theorem \ref{thm:k-integr}), which is important  for  later deriving   the classical   Cram\'er-Rao inequalities   from our  general  Cram\'er-Rao  inequality.
Finally we introduce  the notion of {\it essential tangent space}
of  a 2-integrable parametrized  measure  model (Definition \ref{def:ess})  and  the  related notion of {\it reduced Fisher metric}.

\subsection{A characterization of $k$-integrable  parametrized  measure  models}\label{subs:density}

Here is the definition of a parametrized measure model from \cite[Definition 4.1]{AJLS2016b}.

\begin{definition} \label{def:param-measure-model}
Let $\Om$ be a measurable space.
\begin{enumerate}
\item
A {\em parametrized measure model } is a triple $(M, \Om, \pb)$ where $M$ is a (finite or infinite dimensional) Banach manifold and $\pb: M \to \Mm(\Om) \subset \Ss(\Om)$ is a Frech\'et-$C^1$-map, which we shall  call simply a $C^1$-map.
\item
The triple $(M, \Om, \pb)$ is called a {\em statistical model } if it consists only of probability measures, i.e., such that the image of $\pb$ is contained in $\Pp(\Om)$.
\item
We call such a model {\em dominated by $\mu_0$} if the image of $\pb$ is contained in $\Ss(\Om, \mu_0)$. In this case, we use the notation $(M, \Om, \mu_0, \pb)$ for this model.
\end{enumerate}
\end{definition}

\begin{remark}\label{rem:dom}  In classical mathematical statistics  the existence of  a dominating  measure  for a statistical  model  is an essential  requirement, see  e.g.   Condition $A_\mu$ in \cite[p. 67]{Borovkov1998}. Under this condition    important      notions    e.g.  Fisher metric, Kullback-Leibler divergence,   and MLE
have  been introduced   and the  estimation problem  of   probability  measures  is    called   the problem of  probability density estimation.
\end{remark}

The existence of a dominating measure $\mu_0$ is not a strong restriction, as the following shows.

\begin{proposition} \label{prop:mfld-dom}
Let $(M, \Om, \pb)$ be a parametrized measure model. If $M$ contains a countable dense subset, e.g., if $M$ is a finite dimensional manifold, then there is a measure $\mu_0 \in \Mm(\Om)$ dominating all measures $\pb(\xi)$.
\end{proposition}

\begin{proof} For the proof, we first observe that for a countable family $\{ \nu_n\; : \; n \in \N\}\subset \Ss(\Om)$ of signed measures, the measure
\[
\mu_0 := \sum_{n=1}^\infty \frac1{2^n \|\nu_n\|_{TV}} |\nu_n|.
\]
dominates all $\nu_n$. Let $(\xi_n)_{n \in \N} \subset M$ be a countable dense subset and let $\mu_0 \in \Mm(\Om)$ dominate all $\pb(\xi_n)$. As the inclusion $\Ss(\Om, \mu_0) \hookrightarrow \Ss(\Om)$ is an isometry and hence has a closed image,  we have $\pb(M) \subset \overline{\Ss(\Om, \mu_0)} = \Ss(\Om, \mu_0)$ by the continuity of $\pb$.
\end{proof}

If the measures  $\pb(\xi)$,  $\xi \in M$,  are dominated by $\mu_0$, then we write
\begin{equation}
\pb (\xi) = p(\xi) \mu_0 \text{   for  some } p(\xi) \in L^1 (\Om, \mu_0).\label{eq:density}
\end{equation}

\begin{definition}\label{def:reg}(\cite[Definition 4.2]{AJLS2016b})	We say that the model  $(M, \Om, \pb)$ has a {\em regular density function} if the density function $p: \Om \times M \to \R$ satisfying (\ref{eq:density}) can be chosen such that for all $v \in T_\xi M$ the partial derivative $\partial_v p(.;\xi)$ exists and lies in $L^1(\Om, \mu_0)$ for some fixed $\mu_0$.
\end{definition}

For a  parametrized measure model $(M, \Om, \pb)$, the differential $d_\xi \pb(v)$ for $v \in T_\xi M$ is dominated by $\pb (\xi)$ \cite[Proposition 2.1]{AJLS2016b}, and we may thus define the {\it logarithmic derivative of $\pb$ at $\xi$ in direction $v$} as
\begin{equation}
\p_ v \log \pb  (\xi) : = \frac{d\{d_\xi \pb(v)\}}{d\pb(\xi)} \in L^1(\Om, \pb(\xi)). \label{eq:log}
\end{equation}

\begin{remark} \label{rem:reg}
The standard notion of a statistical model always assumes that it is dominated by some measure and has a positive regular density function (e.g. \cite[p. 140, p.147]{Borovkov1998}, \cite[p. 23]{BKRW1998},\cite[\S 2.1]{AN2000},  \cite[Definition 2.4]{AJLS2015}). In fact, the definition of a parametrized measure model or statistical model in \cite[Definition 2.4]{AJLS2015} is equivalent to a parametrized measure model or statistical model with a positive regular density function in the sense of Definition \ref{def:reg}.
\end{remark} 
	
If  the model  has a positive  regular  density  function, then we have
\begin{equation}
\p_v \log \pb  (\xi) = \p_v \log p,\label{eq:derlogreg}
\end{equation}
i.e., the logarithmic derivative from (\ref{eq:log}) coincides with the derivative of the logarithm of the density function $p$, justifying the notation from (\ref{eq:log}).

Next  we recall the notion of $k$-integrability introduced in \cite{AJLS2016b}. For this, we define for each $r \in (0, 1]$ the Banach lattice
\begin{equation}
\Ss^r(\Om) := \lim_{\longrightarrow} L^{1/r}(\Om, \mu),
\end{equation}
where the directed limit is taken over the directed set $(\Mm(\Om), \leq)$, where $\mu_1 \leq \mu_2$ if $\mu_2$ dominates $\mu_1$, using the isometric inclusions
\[
\imath_{\mu_2}^{\mu_1}: L^{1/r}(\Om, \mu_1) \longrightarrow L^{1/r}(\Om, \mu_2), \qquad \phi \longmapsto \phi\; \left(\frac{d\mu_1}{d\mu_2}\right)^r.
\]
We denote the element of $\Ss^r(\Om)$ represented by $\phi \in L^{1/r}(\Om, \mu)$ as $\phi \mu^r$,  which allows us to work within $\Ss^r(\Om)$ in a very suggestive way, using the identity
\[
\mu_1^r = \left(\frac{d\mu_1}{d\mu_2}\right)^r \mu_2^r
\]
for $\mu_1 \leq \mu_2$. Since $\imath_{\mu_2}^{\mu_1}$ is an isometry, $\Ss^r(\Om)$ inherits a Banach norm, denoted by $\|.\|_{\Ss^r(\Om)}$, such that the inclusion $L^{1/r}(\Om, \mu_0) \hookrightarrow \Ss^r(\Om)$, $\phi \mapsto \phi \mu^r_0$ becomes an isometry, whose image is denoted by $\Ss^r(\Om, \mu_0)$.

There is a bilinear continuous multiplication map
\[
\cdot:  \Ss^r(\Om) \times \Ss^r(\Om) \longrightarrow \Ss^{r+s}(\Om), \qquad (\phi \mu^r) \cdot (\psi \mu^r) := \phi \psi \mu^{r+s}
\]
for $r, s, r+s \in (0, 1]$. Furthermore, for  $r \in (0, 1]$ and  $0< k \le  1/r$  we define the map
\[
\pi ^k : \Ss^r(\Om) \to  \Ss^{rk} (\Om), \qquad  \phi \cdot \mu ^r  \mapsto {\rm sign}  (\phi) | \phi | ^k \mu ^{rk}.
\]
This map is continuous for all $k$, and it is Fr\'echet-differentiable for $k \geq 1$ with derivative
\[
d_{\mu_r} \pi^k (\eta_r) = k \pi^{k-1} |\mu_r| \cdot \eta_r.
\]

\begin{definition} \label{def:k-integra}
A parametrized measure model $(M, \Om, \pb)$ (statistical model, respectively) is called {\em $k$-integrable } if the map 
\[
\pi^{1/k} \pb =: \pb^{1/k}: M \longrightarrow \Mm^{1/k}(\Om) \subset \Ss^{1/k}(\Om)
\]
is a Frech\'et-$C^1$-map.

Observe that $\pb = \pi^k \pb^{1/k}$, whence the chain rule for Fr\'echet differentiable maps implies that the Frech\'et-derivative of $\pb^{1/k}$ is given as
\begin{equation} \label{eq:formal-derivative}
d_\xi \pb^{1/k}(v) := \frac{1}{k} \partial_v \log \pb(\xi) \; \pb^{1/k}(\xi) \in \Ss^{1/k}(\Om, \pb(\xi)).
\end{equation}
\end{definition}

The reader who is familiar with the references \cite{AJLS2015} and \cite{AJLS2016b} will observe that the definitions of $k$-integrability in those references are different from Definition \ref{def:k-integra}. However, as we shall show now, all these notions are equivalent.

\begin{theorem} \label{thm:k-integr}
Let $(M, \Om, \pb)$ be a parametrized measure model and $k > 1$. Then the following are equivalent.
\begin{enumerate}
\item The model is $k$-integrable.
\item For all $v \in T_\xi M$, $\partial_v \log \pb(\xi) \in L^k(\Om,  \pb (\xi))$, and the map
\[
d\pb^{1/k}: TM \longrightarrow \Ss^{1/k}(\Om)
\]
in (\ref{eq:formal-derivative}) is continuous.
\item For all $v \in T_\xi M$, $\partial_v \log \pb(\xi) \in L^k(\Om,  \pb (\xi))$, and the map
\begin{equation} \label{eq:norm-continuous}
v \longmapsto \|\partial_v \log \pb(\xi)\|_{L^k(\Om,  \pb (\xi))} = \| \partial_v \log \pb(\xi) \; \pb^{1/k}(\xi)\|_{\Ss^{1/k}(\Om)}
\end{equation}
is continuous.
\end{enumerate}
\end{theorem}

\begin{proof}
Evidently, if $\pb^{1/k}$ is Fr\'echet-$C^1$, then its derivative $d\pb^{1/k}$ is continuous by definition, whence the first statement implies the second. Moreover,
\[
\| \partial_v \log \pb(\xi) \; \pb^{1/k}(\xi)\|_{\Ss^{1/k}(\Om)} = k \| d_\xi \pb(v)\|_{\Ss^{1/k}(\Om)},
\]
by (\ref{eq:formal-derivative}), so  evidently, the second statement implies the third. Thus, we have to show the converse.

Suppose that the map (\ref{eq:norm-continuous}) is continuous, let $(v_n)_{n \in \N}$ be a sequence, $v_n \in T_{\xi_n}M$ with $v_n \to v_0 \in T_{\xi_0}M$, and let $\mu_0 \in \Mm(\Om)$ be a measure dominating all $\pb(\xi_n)$, which exists by Proposition \ref{prop:mfld-dom}. Multiplying $\mu_0$ with a positive function in $L^1(\Om, \mu_0)$, we may assume that there is a decomposition $\Om = \Om_0 \dot \cup \Om_1$ such that
\[
\pb(\xi_0) = \chi_{\Om_0} \mu_0. 
\]
Let $p_n, q_n \in L^1(\Om, \mu_0)$ such that $\pb(\xi_n) = p_n \mu_0$ and $d\pb(v_n) = q_n \mu_0$, so that
\begin{equation} \label{eq:q_nk}
d\pb^{1/k}(v_n) = \frac{q_n}{k p_n^{1-1/k}} \chi_{\{p_n > 0\}} \mu_0^{1/k} =: q_{n;k} \mu_0^{1/k}.
\end{equation}
In particular, $p_0 = \chi_{\Om_0}$, and $\|d\pb^{1/k}(v_n)\|_{S^{1/k}(\Om)} = \|q_{n;k}\|_k$, and by the continuity of (\ref{eq:norm-continuous}) it follows that
\begin{equation} \label{eq:k-norm-conv}
\lim \|q_{n;k}\|_k = \|q_{0;k}\|_k,
\end{equation}
where  $\|\cdot\|_k$ denotes the norm in $L^k(\Om, \mu_0)$.
On $\Om_0$ we estimate
\begin{align*}
|q_{n;k} - q_{0;k}| & = \left| \dfrac{q_n}{k p_n^{1-1/k}} - \dfrac{q_0}k\right| \leq \dfrac1k |q_n - q_0| + |q_{n;k}| \left|1 - p_n^{1-1/k}\right|\\
& \leq \dfrac1k |q_n - q_0| + |q_{n;k}| |1 - p_n|^{1-1/k}.
\end{align*}
Thus, since $p_0 = \chi_{\Om_0}$ and $q_0, q_{0;k}$ vanishes on $\Om_1$, we have by H\"older's inequality
\begin{align*}
\|\chi_{\Om_0} q_{n;k} - q_{0;k}\|_1 & \leq \dfrac1k \|q_n - q_0\|_1 + \|q_{n;k}\|_k \|p_n - p_0\|_1^{1-1/k}\\
& = \dfrac1k \|\p_{v_n} \pb - \p_{v_0} \pb\|_1 + \|q_{n;k}\|_k \|\pb(\xi_n) - \pb(\xi_0)\|_1^{1-1/k}.
\end{align*}
Since $\pb$ is a $C^1$-map, both $\|\p_{v_n} \pb - \p_{v_0} \pb\|_1$ and $\|\pb(\xi_n) - \pb(\xi_0)\|_1$ tend to $0$, whereas $\|q_{n;k}\|_k$ is bounded by (\ref{eq:k-norm-conv}). Thus, $\chi_{\Om_0} q_{n;k} \to q_{0;k}$ in $L^1(\Om, \mu_0)$, and as $\|\chi_{\Om_0} q_{n;k}\|_k \leq \|q_{n;k}\|_k$ is bounded, this implies that
\begin{equation} \label{eq:chi0-weak}
\chi_{\Om_0} q_{n;k} \rightharpoonup q_{0;k} \qquad \mbox{in $L^k(\Om, \mu_0)$}.
\end{equation}
This weak convergence implies that
\[
\|q_{0;k}\|_k  \leq \liminf \|\chi_{\Om_0} q_{n;k}\|_k \leq \limsup \|q_{n;k}\|_k \stackrel{(\ref{eq:k-norm-conv})} = \|q_{0;k}\|_k,
\]
so that we have equality in these estimates, and hence,
\[
\lim \|\chi_{\Om_0} q_{n;k}\|_k^k = \lim \|q_{n;k}\|_k^k = \lim \|\chi_{\Om_0} q_{n;k}\|_k^k + \lim \|\chi_{\Om_1} q_{n;k}\|_k^k
\]
which means that $\chi_{\Om_1} q_{n;k} \to 0$ in $L^k(\Om, \mu_0)$. Thus, (\ref{eq:chi0-weak}) now implies that $q_{n;k} \rightharpoonup q_{0;k}$ in $L^k(\Om, \mu_0)$, and by the Radon-Riesz theorem, this together with (\ref{eq:k-norm-conv}) implies that $\|\p_{v_n}\pb^{1/k} - \p_{v_0}\pb^{1/k}\|_{S^{1/k}(\Om, \mu_0)} = \|q_{n;k} - q_{0;k}\|_k \to 0$, i.e., $\lim \p_{v_n}\pb^{1/k} = \p_{v_0}\pb^{1/k}$ in $S^{1/k}(\Om, \mu_0)$ and hence, the continuity of $d\pb^{1/k}$ follows.

Thus, we have shown that the third statement of the theorem implies the second.

Now let us assume that the map $d\pb: TM \to \Ss^{1/k}(M)$ is continuous, and let $\xi: I \to M$ be a curve. By Proposition \ref{prop:mfld-dom}, there is a finite measure $\mu_0$ dominating $\pb(\xi_t)$ for all $t \in I$. In order to be able to divide by powers of our measures, we define $\pb_\eps(\xi) := \pb(\xi) + \eps \mu_0$ for $\eps \geq 0$, so that $(M, \Om, \pb_\eps)$ is again a parametrized measure model, and $\pb = \pb_0$. As before, we define $p_t^\eps = p_t + \eps, q_t^\eps = q_t \in L^1(\Om, \mu_0)$ such that $\pb_\eps(\xi_t) = p_t^\eps \mu_0$ and $d\pb_\eps(\dot \xi_t) = q_t^\eps \mu_0 = q_t \mu_0$, so that
\[
d\pb_\eps^{1/k}(\dot \xi_t) = \frac{q_t}{k (p_t^\eps)^{1-1/k}} \mu_0^{1/k} =: q^\eps_{t;k} \mu_0^{1/k}.
\]
Furthermore, we define for each $l \geq 1$ and $t, t_0 \in I$ the remainder term
\begin{align*}
& r^\eps_{t,t_0;l} := (p^\eps_{t+t_0})^{1/l} - (p^\eps_{t_0})^{1/l} - t q_{t_0;l} \in L^l(\Om, \mu_0)\\
\Rightarrow \qquad  & r^\eps_{t,t_0;l}\; \mu_0^{1/l} = \pb_\eps^{1/l}(\xi_{t+t_0}) - \pb_\eps^{1/l}(\xi_{t_0}) - t d\pb_\eps^{1/l}(\dot \xi_{t_0}).
\end{align*}

For $\eps > 0$, by the mean value theorem, there is an $\eta_t$ between $p_{t+t_0}^\eps$ and $p_{t_0}^\eps$ (and hence, $\eta_t \geq \eps$) for which
\begin{align*}
|r_{t,t_0;k}^\eps| &= \left|(p_{t+t_0}^\eps)^{1/k} - (p_{t_0}^\eps)^{1/k} - t q_{t_0;k}^\eps\right| = \left|\frac{p_{t+t_0} - p_{t_0}}{k \eta_t^{1-1/k}} - t q_{t_0;k}^\eps\right|\\
&\leq \frac{|r^0_{t,t_0;1}|}{k \eta_t^{1-1/k}} + |t| \left| \frac{q_{t_0}}{k \eta_t^{1-1/k}} - q_{t_0;k}^\eps\right|\\
& = \frac{|r^0_{t,t_0;1}|}{k \eta_t^{1-1/k}} + |t| |q_{t_0;k}^\eps| \frac{|(p_{t_0}^\eps)^{1-1/k} - \eta_t^{1-1/k}|}{\eta_t^{1-1/k}}\\
& \leq \frac1{k \eps^{1-1/k}} \left(|r^0_{t,t_0;1}| + k |t| |q_{t_0;k}^\eps| |(p_{t_0}^\eps) - \eta_t|^{1-1/k}\right)\\
& \leq \frac1{k \eps^{1-1/k}} \left(|r^0_{t,t_0;1}| + k |t| |q_{t_0;k}| |p_{t_0} - p_{t+t_0}|^{1-1/k}\right).
\end{align*}
Integration and H\"older's inequality yields
\[
\|r_{t,t_0;k}^\eps\|_1 \leq \frac1{k \eps^{1-1/k}} \left(\|r^0_{t,t_0;1}\|_1 + k |t| \|q_{t_0;k}\|_k \|p_{t_0} - p_{t+t_0}\|_1^{1-1/k}\right).
\]
Since $t^{-1} \|r^0_{t,t_0;1}\|_1 = t^{-1} \|\pb(\xi_{t+t_0}) - \pb(\xi_{t_0}) - t d\pb(\dot \xi_{t_0})\|_{\Ss(\Om)} \xrightarrow{t\to0}0$ and $\|p_{t_0} - p_{t+t_0}\|_1 = \|\pb(\xi_{t+t_0}) - \pb(\xi_{t_0})\|_{\Ss(\Om)}\xrightarrow{t\to0}0$, as $\pb$ is Fr\'echet differentiabe, it follows that for $\eps > 0$,
\begin{equation} \label{eq:remainder-L1}
\lim_{t \to 0} \frac1t \|r_{t,t_0;k}^\eps\|_1 = 0.
\end{equation}
Moreover, we can make the following estimate:
\begin{align*}
|q_{t;k}^\eps - q_{t_0;k}^\eps| &= \left| \frac{q_t}{k(p_t^\eps)^{1-1/k}} - \frac{q_{t_0}}{k(p_{t_0}^\eps)^{1-1/k}}\right|\\
& \leq \frac1{k(p_{t_0}^\eps)^{1-1/k}} |q_t - q_{t_0}| + \left| \frac1{k(p_t^\eps)^{1-1/k}} - \frac1{k(p_{t_0}^\eps)^{1-1/k}}\right|\; |q_t|\\
& \leq \frac1{k(p_{t_0}^\eps)^{1-1/k}} \left(|q_t - q_{t_0}| + k \left| (p_{t_0}^\eps)^{1-1/k} - (p_t^\eps)^{1-1/k}\right|\; |q^\eps_{t;k}|\right)\\
& \leq \frac1{k \eps^{1-1/k}} \left(|q_t - q_{t_0}| + k |p_{t_0} - p_t|^{1-1/k}\; |q_{t;k}|\right)
\end{align*}
Integration and H\"older's inequality implies
\begin{align*}
\|q_{t;k}^\eps - q_{t_0;k}^\eps\|_1 &\leq \frac1{k \eps^{1-1/k}} \left(\|q_t - q_{t_0}\|_1 + k \|p_{t_0} - p_t\|_1^{1-1/k}\; \|q_{t;k}\|_k\right),
\end{align*}
and since both $\|p_t - p_{t_0}\|_1 = \|\pb(\xi_t) - \pb(\xi_{t_0}\|_{\Ss(\Om)}$ and $\|q_t - q_{t_0}\|_1 = \|d\pb(\dot \xi_t) - d\pb(\dot\xi_{t_0})\|_{\Ss(\Om)}$ tend to $0$ for $t \to t_0$ as $\pb$ is a $C^1$-map, it follows that
\begin{equation} \label{eq:dp-contL1}
\lim_{t \to t_0} \|q_{t;k}^\eps - q_{t_0;k}^\eps\|_1 = 0.
\end{equation}

For $f \in L^\infty(\Om)$ consider the function $\tilde f: I \to \R$, 
\[
\tilde f(t) := \int_\Om (p_t^\eps)^{1/k} f\; d\mu_0.
\]
Then (\ref{eq:remainder-L1}) implies that
\[
\tilde f'(t) = \int_\Om q_{t;k}^\eps f\; d\mu_0,
\]
and (\ref{eq:dp-contL1}) implies that $\tilde f'$ is continuous, so that by the fundamental theorem of calculus we have for all $t_0, t_1 \in I$
\begin{equation} \label{eq:int-eq}
\tilde f(t_1) - \tilde f(t_0) = \int_\Om ((p_{t_1}^\eps)^{1/k} -(p_{t_0}^\eps)^{1/k}) f\; d\mu_0 = \int_{t_0}^{t_1} \int_\Om q_{s;k}^\eps f\; d\mu_0\; ds.
\end{equation}
Since $|(p_{t_1}^\eps)^{1/k} -(p_{t_0}^\eps)^{1/k}| \leq |p_{t_1} -p_{t_0}|^{1/k}$ and $|q_{s;k}^\eps| \leq |q_{s;k}|$, we may apply the dominated convergence theorem to (\ref{eq:int-eq}) to conclude that
\begin{equation} \label{eq:int-eq0}
\int_\Om (p_{t_1}^{1/k} - p_{t_0}^{1/k}) f\; d\mu_0 = \int_{t_0}^{t_1} \int_\Om q_{s;k} f\; d\mu_0\; ds
\end{equation}
for any $f \in L^\infty(\Om)$. Now both sides of (\ref{eq:int-eq0}) may be regarded as bounded linear functionals for $f \in L^{k/(k-1)}(\Om, \mu_0)$, since $\|p_{t_1}^{1/k} - p_{t_0}^{1/k}\|_k \leq \|p_{t_1} - p_{t_0}\|_1^{1/k} < \infty$ and $\|q_{s;k}\|_k = \|d\pb(\dot \xi_s)\|_{\Ss^{1/k}(\Om)}$ depends continuously on $s$ by the continuity of $d\pb$. Therefore, since $L^\infty(\Om) \subset L^{k/(k-1)}$ is dense, it follows that (\ref{eq:int-eq0}) holds for all $f \in L^{k/(k-1)}(\Om)$. Thus, for all such $f$ we have
\begin{align*}
\left|\int_\Om r_{t,t_0;k} f\; d\mu_0\right| &= \left|\int_\Om (p_{t+t_0}^{1/k} - p_{t_0}^{1/k} - t q_{t_0;k}) f\; d\mu_0\right|\\
&\stackrel{(\ref{eq:int-eq0})} = \left|\int_{t_0}^{t_0+t} \int_\Om (q_{s;k} - q_{t_0;k}) f\; d\mu_0\; ds\right|\\
& \leq \left|\int_{t_0}^{t_0+t} \|q_{s;k} - q_{t_0;k}\|_k \|f\|_{k/(k-1)}\; ds\right|\\
& \leq |t|  \|f\|_{k/(k-1)} \sup_{|s-t_0| \leq t} \|q_{s;k} - q_{t_0;k}\|_k.
\end{align*}
Now by the Hahn-Banach theorem, we may choose $f \in L^{k/(k-1)}(\Om, \mu_0)$ such that $\int_\Om r_{t,t_0;k} f\; d\mu_0 = \|r_{t,t_0;k}\|_k$ and $\|f\|_{k/(k-1)} = 1$. Then we conclude from this estimate
\[
\|r_{t,t_0;k}\|_k \leq |t| \sup_{|s-t_0| \leq t} \|q_{s;k} - q_{t_0;k}\|_k,
\]
which translates into
\begin{align} \label{eq:est-r1k}
\| \pb(\xi_{t_0+t})^{1/k} - \pb(\xi_{t_0})^{1/k} &- t d\pb^{1/k}(\dot \xi_{t_0}) \|_{\Ss^{1/k}(\Om)}\\
\nonumber & \leq |t| \sup_{|s-t_0| \leq t} \|d\pb^{1/k}(\dot \xi_s) - d\pb^{1/k}(\dot \xi_{t_0})\|_{\Ss^{1/k}(\Om)}
\end{align}
for any curve $(\xi_t)$ in $M$, and this together with the continuity of $d\pb^{1/k}$ implies that $\pb^{1/k}$ is Fr\'echet differentiable. That is, the second statement in Theorem \ref{thm:k-integr} implies the first.
\end{proof}

\begin{remark}\label{rem:fisher} The Fisher metric $\g$ on  a  parametrized measure model $(M, \Om, \pb)$ is defined by
\begin{equation}
\g_\xi(v, w) : = \la \p _v  \log \pb ; \p _w \log \pb \ra_{L^2 (\Om , \pb (\xi))} = \la d\pb^{1/2}(v); d\pb^{1/2}(w)\ra_{\Ss^{1/2}(\Om)}.\label{eq:fisher}
\end{equation}
Thus  the Fisher  metric is well-defined and continuous iff  $(M, \Om, \pb)$ is 2-integrable.
\end{remark}
\subsection{Essential tangent space  and  reduced Fisher metric}\label{subs:reduced}
Let $(M, \Om, \pb)$ be a 2-integrable  parametrized  measure model.    Formula  (\ref{eq:fisher}) shows that  the kernel of the Fisher
metric $\g$  at $ \xi \in  M$ coincides  with the kernel of  the map $\Lambda_\xi : T_\xi  M \to  L^2(\Om, \pb (\xi)), \:  v \mapsto  \p_v (\log  \pb )$.
In other  words, the degeneracy  of the Fisher metric $\g$ is caused  by the non-effectiveness of the parametrization of  the family  $\pb (\xi)$  by  the map $\pb$.
The tangent   cone $T_{\pb (\xi)} \pb (M)$  of the image  $\pb (M) \subset   \Ss (\Om)$  is   isomorphic to the  quotient  $ T_\xi  M / \ker \Lambda _\xi$.  This motivates  the following 
\begin{definition}\label{def:ess}
	The   quotient $\hat  T_\xi M: =  T_\xi  M / \ker \Lambda _\xi$  will be called the {\it   essential tangent  space  of $M$   at $\xi$}.
\end{definition}
Clearly, the  Fisher metric   $\g$ descends
to   a non-degenerated   metric  $\hat \g$  on $\hat  T_\xi M$, which we shall  call the {\it  reduced  Fisher metric}.  
Denote  by $\hat  T^{\hat \g}M$ the fiberwise completion of $\hat TM$ w.r.t.  the reduced Fisher metric $\hat\g$. Its  inverse   $ \hat\g ^{-1}$  is a well-defined   quadratic   form on the fibers of  the dual bundle   
$\hat T^{*, \hat \g ^{-1}}M$
which  we can therefore identify with  $\hat T^{\hat \g}M$.

\begin{remark}\label{rem:completion}  The  fiberwise completion  $\hat T^{\hat \g}M$  is different   from  $\hat TM$  only if  $M$ is  infinite dimensional. Observe that the map $\hat T^{\hat \g}M \to M$ is not a fiber bundle in general, as we do not define a topology on the total space $\hat T^{\hat \g}M$. Nevertheless, we shall call the left inverses of this map {\em sections of $\hat T^{\hat\g}M$}.
\end{remark}

\begin{example}\label{ex:normalm}  One  of  the typical  singular   statistical models   considered  in \cite[Example 1.2, p. 14]{Watanabe2009} is      the  normal  mixture  family $(W, \R,  dx, p)$
	where
	$$W =\{ ( a, b) \in \R^2 |\,  a \in [0,1],  b \in \R \} $$
	$$ p ( x|  a, b ) : = \frac{(1 -a)  e^{ - x^2 /2}  + a   e ^{ - (x -b) ^2 /2}}{\sqrt {2\pi}} .$$
	This  family is a typical  example   of  Gaussian mixture models  which comprise  also  the changing time  model  (the Nile River model) and the ARMA model  in time series \cite[\S 12.2.6, p. 311]{Amari2016}.  We compute
	$$\p _a p (x | a, b) = \frac{-  e^{ - x^2 /2}  +  e ^{ - (x -b) ^2 /2}}{\sqrt {2\pi}} , $$
	$$  \p _b p (x | a, b) = \frac{ a (x-b)  e ^{ - (x -b) ^2 /2}}{\sqrt {2\pi}} . $$
	Hence  $ \p _a  p(x |  a, b) = 0 \, \forall x $ iff  $b = 0$   and  $\p _b p (x|  a, b)  = \, \forall x$ iff  $ a = 0$.  Furthermore  it is not hard to see  that  
	$(\p_a  p (x|   a, b)$ and $ \p _b  p (x| a, b)) $ are linearly independent. Thus  the   singularity  of   $(W, \R, dx, p)$ is $\{  a = 0\} \cup \{  b = 0\}$.
	Furthermore  $\hat T_{(0, 0)} W = \{ pt\}$,   $\hat T_{(a, 0)}W = \R^2 / (\R, 0)$  for $ a\not = 0$, $\hat T_{(0, b)} W  = \R^2 /(0, \R)$  for $b \not = 0$.
\end{example}

\section{Visible  functions, their    generalized gradient and pre-gradient}\label{sec:vis}
Motivated by  problems of  parameter estimation  in mathematical statistics and machine learning,  we introduce the notion of  a {\it regular function on $\Om$} (Definition \ref{def:regular-k}),  a visible function  on $M$  (Definition  \ref{def:obs})
and its   generalized gradient and pre-gradient (Definitions \ref{def:grad}, \ref{def:pregrad}). Our main results in this section are  Propositions \ref{prop:change-diff-int}, \ref{prop:cov}. The first one  asserts the validity of differentiation under integral sign, which is important  for the proof of  the second one  that
asserts the existence of the pre-gradient  of  functions  associated  to  {\it $\varphi$-regular} parameter estimators   in  statistical inference.  

Finally, in Subsection \ref{subs:finite} we apply the obtained results  to the   parametrized measure model of all  measures (resp. probability measures) on a finite sample  space.

\subsection{Visible  functions  and  estimators}\label{subs:obs}

Given an parametrized measure model $(M, \Om, \pb)$, we  set  for $k \geq 1$
$$L^k_M (\Om): = \{\varphi: \Om \to \R \mid   \varphi\in  L^k  (\Om, \pb (\xi)) \, \text{ for  all  }  \xi \in M\}.$$  

For $\varphi \in L^k_M (\Om)$ we obtain a map $\varphi \pb^{1/k}: M \to \Ss^{1/k}(\Om)$, $\xi \mapsto \varphi \pb(\xi)^{1/k}$. In general, we cannot expect $\varphi \pb^{1/k}$ to be differentiable, not even continuous, as the following example illustrates.

\begin{example}
Let $\Om := (-1,1)$ and let $h: \R \to \R$ be a $C^\infty$-function with $h(x) > 0$ for $x \in (0,1)$ and $h(x) = 0$ for $x \notin(0,1)$, and such that $\int_\R h(x)\; dx = 1$. Let $\alpha > 1$ and $\beta > 0$ be fixed, and define the family $(\pb(t))_{t \in (-1,1)}$ on $\Om$ by
\begin{align*}
\pb_t &= \Big((1 - |t|^{\alpha+1}) \chi_{(-1,0)} + |t|^\alpha h(|t|^{-1} x) \chi_{(0, 1)} \Big)\; dx, \; t \neq 0,\\
\pb_0 &= \chi_{(-1,0)}\; dx.
\end{align*}
The density function on $(-1,0)$ is chosen such that $\pb_t$ is a probability measure on $\Om$ for all $t$. Then $d\pb_0 = 0$, and for $t \neq 0$,
\[
d\pb_t = \sgn(t) \Big(-(\alpha + 1) |t|^\alpha\chi_{(-1,0)} + |t|^{\alpha-1} g(|t|^{-1} x) \chi_{(0,1)}\Big)\; dx,
\]
where
\[
g(x) := \alpha h(x) - x h'(x),
\]
and it is straightforward to see that $\|d\pb_t - d\pb_{t_0}\|_1 \to 0$ as $t \to t_0$, so that $\pb$ is a parametrized measure model.

Observe that for any $l \geq 1$
\begin{align*}
l^l \|\p_t \pb^{1/l}\|_{\Ss^{1/l}(\Om)}^l & \stackrel{(\ref{eq:formal-derivative})}= \left\|\frac{d\pb_t}{\pb_t^{1-1/l}}\right\|_{L^l(\Om, dx)}^l\\
& = \int_{-1}^0 \frac{((\alpha + 1) |t|^\alpha)^l}{(1 - |t|^{\alpha+1})^{l-1}}\; dx + \int_0^1 \frac{(|t|^{\alpha-1} g(|t|^{-1} x))^l}{(|t|^\alpha h(|t|^{-1} x))^{l-1}}\; dx\\
& = \frac{(\alpha + 1)^l |t|^{l \alpha}}{(1 - |t|^{\alpha+1})^{l-1}} + |t|^{\alpha + 1 - l} \int_0^1 \frac{g(u)^l}{h(u)^{l-1}}\; du,
\end{align*}
using the substitution $u = t^{-1} x$. Observe that
\[
\left|\int_0^1 \frac{g(u)^l}{h(u)^{l-1}}\; du\right|^{1/l} = \|\alpha h^{1/l} - l u (h^{1/l})'\|_{L^l(\Om, dx)} < \infty,
\]
since $h^{1/l}$ is smooth as $h$ vanishes to infinite order at $u = 0$. Thus, if $l < \alpha + 1$, then $\|\p_t \pb^{1/l}\|_{\Ss^{1/l}(\Om)}$ depends continuously on $t$ and therefore, by Theorem \ref{thm:k-integr}, $\pb$ is $l$-integrable for all $l < \alpha + 1$.

Now let $\varphi(x) := \chi_{(0, 1)} x^{-\beta}$. Then for any $k > 1$, $\|\varphi\|_{L^k(\Om, \pb_0)} = 0$, and for $t \neq 0$ we have
\begin{align*}
\|\varphi\|_{L^k(\Om, \pb_t)}^k & = \int_0^1 x^{-k\beta} |t|^\alpha h(|t|^{-1} x)\; dx\\
& = |t|^{\alpha + 1 - k \beta} \int_0^1 u^{-k \beta} h(u)\; du < \infty,
\end{align*}
and therefore,
\[
\varphi \in L^k_{(-1,1)}(\Om) \quad \mbox{for all $k \geq 1$}.
\]
On the other hand,
\[
\E_{\pb_t}(\varphi) = \int_0^1 x^{-\beta} |t|^\alpha h(|t|^{-1} x)\; dx = |t|^{\alpha + 1 - \beta} \int_0^1 u^{-\beta} h(u)\; du,
\]
so that for $\beta > \alpha + 1$ we have $\lim_{t \to 0} \E_{\pb_t}(\varphi) = \infty$.

That is, for a given $l > 1$ choosing the parameters such that $\beta > \alpha + 1 > l$, $((-1,1), \Om, \pb)$ is an $l$-integrable model, $\varphi \in L^k_{(-1,1)}(\Om)$ for all $k \geq 1$, but the function $t \mapsto \E_{\pb_t}(\varphi)$ is discontinuous.
\end{example}

Observe that the failure of the map $t \mapsto \E_{\pb_t}(\varphi)$ in the preceding example to be continuous at $t = 0$ is due to the unboundedness of the map $t \mapsto \|\varphi\|_{L^k(\Om, \pb(t))}$. This motivates the following definition.

\begin{definition} \label{def:regular-k}
Let $(M, \Om, \pb)$ be a parametrized measure model. We call a function $\varphi$ on $\Om$ {\em $k$-regular}, if $\varphi\in L^k_M(\Om)$ and moreover if the function $\xi \mapsto \|\varphi\|_{L^k(\Om, \pb(\xi))}$ is locally bounded, i.e. if for all $\xi_0 \in M$
\[
\limsup_{\xi \to \xi_0} \|\varphi\|_{L^k(\Om, \pb(\xi))} < \infty.
\]
\end{definition}
If there is no danger of confusion, we shall call a $k$-regular function $\varphi$ simply a   
{\it regular  function}.

\begin{proposition} \label{prop:change-diff-int}
	Let $k, k' > 1$ be dual indices, i.e. $k^{-1} + {k'}^{-1} = 1$, and let $(M, \Om, \pb)$ be an $k'$-integrable parametrized measure model. If $\varphi \in L^k_M (\Om)$ is regular, then the map
	\begin{equation} \label{eq:phi-p}
	M \longrightarrow \R, \qquad \xi \longmapsto \E_{\pb(\xi)}(\varphi) = \int_\Om \varphi\; d\pb(\xi)
	\end{equation}
	is Gat\^eaux-differentiable, and for $X \in TM$ the G\^ateaux-derivative is
	\begin{equation} \label{eq:phi-p-deriv}
	\p_X \E_{\pb(\xi)}(\varphi) = \E_{\pb(\xi)}(\varphi \; \p_X \log\pb(\xi)) = \int_\Om \varphi \; \p_X \log \pb(\xi)\; d\pb(\xi).
	\end{equation}
\end{proposition}

\begin{proof} Let $X \in T_{\xi_0}M$, and let $\xi_t$ be a differentiable curve in $M$ with $\dot \xi_0 = X$. By Proposition \ref{prop:mfld-dom}, there is a measure $\mu_0 \in \Mm(\Om)$ which dominates all $\pb(\xi_t)$. In fact, when replacing $\mu_0$ by $(\max\{|\varphi|^{k'}, 1\})^{-1} \mu_0$, we may assume w.l.o.g. that in addition $\varphi \in L^{k'}(\Om, \mu_0)$.
	
	As in the proof of Theorem \ref{thm:k-integr}, we define the functions $p_t, q_t \in L^1(\Om, \mu_0)$ such that $\pb(\xi_t) = p_t \mu_0$ and $d\pb(\dot \xi_t) = \p_t \pb(\xi_t) = q_t \mu_0$.  Also,   let $\|\cdot\|_r$ denote the norm in $L^r(\Om, \mu_0)$. Then by H\"older's inequality
	\begin{align}
	\label{eq:est-varphi1}
	\|\varphi (p_t^{1/k} - p_0^{1/k})\|_1 & \leq \|\varphi\|_{k'} \|p_t^{1/k} - p_0^{1/k}\|_k \leq \|\varphi\|_{k'} \|(p_t - p_0)^{1/k}\|_k\\
	\nonumber
	& = \|\varphi\|_{k'} \|p_t - p_0\|_1^{1/k} \xrightarrow{t \to 0} 0
	\end{align}
	as $\|p_t - p_0\|_1 = \|\pb(\xi_t) - \pb(\xi_0)\|_{\Ss(\Om)} \to 0$. Furthermore, 
	\[\limsup_{t \to 0} \|\varphi p_t^{1/k}\|_k = \limsup_{t \to 0} \|\varphi\|_{L^k(\Om, \pb(\xi_t))} < \infty\]
	by the regularity of $\varphi$, which together with (\ref{eq:est-varphi1}) implies that $\varphi p_t^{1/k} \rightharpoonup \varphi p_0^{1/k}$ in $L^k(\Om, \mu_0)$ and therefore,
	\begin{equation} \label{eq:conv-1}
	\la \varphi p_t^{1/k} - \varphi p_0^{1/k}; q_{0;k'} \ra \xrightarrow {t \to 0} 0,
	\end{equation}
	where $\la f; g\ra := \E_{\mu_0}(fg)$ stands for the canonical dual pairing of $L^k(\Om, \mu_0)$ and $L^{k'}(\Om, \mu_0)$, and where we define
	\[
	q_{t;k'} := \frac{q_t}{k'\; p_t^{1/k}} \chi_{\{p_t > 0\}} \in L^{k'}(\Om, \mu_0), \quad \mbox{so that $d\pb^{1/k'}(\dot \xi_t) = q_{t;k'} \mu_0^{1/k'}$}
	\]
	analogously to (\ref{eq:q_nk}). Furthermore, again by H\"older's inequality,
	\begin{equation} \label{eq:conv-2}
	\la \varphi p_t^{1/k}; q_{t;k'} - q_{0;k'}\ra \leq \|\varphi p_t^{1/k}\|_k \|q_{t;k'} - q_{0;k'}\|_{k'} \xrightarrow{t\to 0} 0,
	\end{equation}
	since 
	\[
	\|q_{t;k'} - q_{0;k'}\|_{k'} = \|d\pb^{1/k'}(\dot \xi_t) - d\pb^{1/k'}(\dot \xi_0)\|_{\Ss^{1/k'}(\Om)} \to 0
	\]
	by the $k'$-integrability of $\pb$ and hence the continuity of $d\pb^{1/k'}$, and since $\|\varphi p_t^{1/k}\|_k = \|\varphi\|_{L^k(\Om, \pb(\xi_t))}$ is bounded by the regularity of $\varphi$.
	
	From (\ref{eq:conv-1}) and (\ref{eq:conv-2}) we now obtain
	\begin{align*}
	\la \varphi p_t^{1/k}; q_{t;k'} \ra - \la \varphi p_0^{1/k}; q_{0;k'} \ra = & \la \varphi p_t^{1/k}; q_{t;k'} - q_{0;k'}\ra \\ & \quad + \la \varphi p_t^{1/k} - \varphi p_0^{1/k}; q_{0;k'} \ra \xrightarrow{t\to 0} 0,
	\end{align*}
	and therefore from the definition of the dual pairing $\la \cdot;\cdot\ra$ and of $q_{t;k'}$, and as $q_t\mu_0 = \p_t \log \pb\; \pb(t)$, we conclude
	\begin{equation} \label{eq:F-limit}
	\lim_{t \to 0} \int_\Om \varphi\; \p_t \log \pb\; \pb(t) = \int_\Om \varphi\; \p_X \log \pb\; d\pb(\xi_0).
	\end{equation}
	Also observe that
	\begin{equation} \label{eq:int-difference}
	\int_\Om \varphi d\pb(\xi_t) - \int_\Om \varphi d\pb(\xi_0) = \int_0^t \int_\Om \varphi \p_t \log \pb|_{t=s}\; d\pb(s)\; ds.
	\end{equation}
	Indeed, (\ref{eq:int-difference}) holds if $\varphi \in L^\infty(\Om)$ is bounded, using (\ref{eq:int-eq0}) for $k = 1$, and an arbitrary $\varphi \in L^k_M(\Om)$ can be monotonically approximated by bounded functions, so that (\ref{eq:int-difference}) follows from the monotone convergence theorem. Thus,
	\begin{align*}
	\left. \frac d{dt}\right|_{t=0} \int_\Om \varphi d\pb(\xi_t) & = \lim_{t \to 0} \frac1t \int_0^t \int_\Om \varphi \p_t \log \pb|_{t=s}\; d\pb(s)\; ds\\
	& = \int_\Om \varphi \p_X \log \pb\; d\pb(\xi_0),
	\end{align*}
	using (\ref{eq:F-limit}) in the last equation, and from this, (\ref{eq:phi-p-deriv}) follows.
\end{proof}

Let  $V$ be a   topological real vector space,  which  may be infinite dimensional. 
We denote by $ V^M$   the vector space  of  all   $V$-valued functions  on $M$.
A $V$-valued   function  $\varphi$ will stand
for the  coordinate  functions  on $M$, or  in general,    a feature  of $M$ (cf. \cite{BKRW1998}).
Let $V^*$ denote  the   dual space of $V$.
For $l\in V^*$ we denote  the composition $l \circ \varphi$ by
$\varphi^l$. This should be considered as the $l$-th coordinate of
$\varphi$. 

Recall that {\it an estimator} is  a map  $\hat \sigma: \Om  \to M$. If $k, k' > 1$ are dual indices, i.e., $k^{-1} + {k'}^{-1} = 1$, and given a $k'$-integrable parametrized measure model $(M, \Om, \pb)$ and a function $\varphi \in V^M$, we define
\begin{equation*}
L^k_\varphi (M, \Om) := \{ \hat \sigma: \Om \to  M \mid \varphi^l\circ \hat \sigma  \in L^k_M (\Om)  \text{ for all } l \in V^*\}.
\end{equation*}
We call an estimator $\hat{\sigma} \in L^k_\varphi (M, \Om)$ {\em $\varphi$-regular }if $\varphi^l\circ \hat \sigma \in L^k_M(\Om)$ is regular for all $l \in V^\ast$.

Any  $\hat \sigma \in L^k_\varphi (M, \Om)$ induces a  $V^{**}$-valued function $\varphi_{\hat \sigma}$  on $M$ by computing the  expectation of the  composition $\varphi \circ \hat \sigma$
as follows
\begin{equation}
\la\varphi_{\hat \sigma} (\xi), l\ra: = \mathbb E_{\pb(\xi)} ( \varphi^l\circ \hat \sigma)= \int_\Om \varphi^l \circ \hat \sigma\, d\pb(\xi)\label{eq:expec}
\end{equation}
for any  $l \in  V^*$. If $\hat \sigma \in L^k_\varphi (M, \Om)$ is $\varphi$-regular, then Proposition \ref{prop:change-diff-int} immediately implies that $\varphi_{\hat \sigma}: M \to V^{\ast \ast}$ is G\^ateaux-differentiable   with G\^ateaux-derivative
\begin{equation} 
\la \p_X \varphi_{ \hat\sigma}(\xi), l\ra  =  \int_\Om\varphi ^l \circ \hat \sigma  \cdot \p _X \log  \pb (\xi)\, \pb (\xi).\label{eq:com1}
\end{equation}

\begin{definition}\label{def:obs}  A  $V$-valued Gateaux-differentiable function $f$ on $M$  is called {\it visible}   if  $ df$ vanishes on $\ker d\pb \subset TM$.
\end{definition}

For instance, the function from (\ref{eq:expec}) is visible.

\begin{example}\label{ex:value}
If $\pb: M \to \Mm(\Om)$ is a $C^1$-immersion, that is, $\ker d\pb = 0$, then evidently, any G\^ateaux-differentiable function $\varphi: M \to V$ into any topological vector space is visible.

A typical example of such a map is used in semi-parametric statistics, where one considers product manifolds $M = P_1 \times P_2$ with $P_1 $ an open subset  of  $\R^n$  and   $P_2$ a subset
of an infinite dimensional Banach space  $B$, see e.g.\cite[p. 2]{BKRW1998}. In this case, one considers the canonical projection $\varphi_1: M = P_1 \times P_2 \to P_1 \subset \R^n$.

\end{example}

\begin{example}\label{ex:vis}  Let $\varphi : \Ss (\Om) \to \R$ be a  $C^1$-differentiable   function. Then  $f : = \varphi \circ  \pb:  M \to \R$  is a visible function.
\end{example}

\begin{example}\label{ex:push}  Most important  visible     functions are associated with   estimators, which are defined  as in (\ref{eq:expec})   and whose G\^ateaux-differentiability is   established by Proposition \ref{prop:change-diff-int}.
\end{example}

\begin{remark}\label{rem:nophi} Classically,  one considers     2-integrable statistical models  $P$ which are  open subsets  in a vector space $V$   with  coordinates $\theta$ 
\cite{BKRW1998, Borovkov1998, CT2006, WMS2008}.
In this case $\theta$   is regarded as  the parameter  of $P$  and  $\varphi$ is the identity  mapping and hence  omitted. Estimators then  are denoted by $\theta^*$, $\hat  \theta$ or $T$.
The  function $\varphi_{\hat \sigma} (\xi)$ in this case, denoted by $E_\theta (\theta^*)$, is the mean value  (w.r.t.the  measure $\theta$) of   the estimator  $\theta^*$ regarded as an element in $V^{**}$.
\end{remark}

\subsection{Generalized gradient  and pre-gradient   of    visible   functions}\label{subs:pregrad}

From this point onward, we shall assume that $(M, \Om, \pb)$ is a $2$-integrable parametrized measure model, so that in particular the Fisher metric $\g$ on $M$ is well defined.

Let $f$  be a visible     function  on  $(M, \Om, \pb)$. Since  $df$ vanishes  on   the kernel of $\pb$, 
the derivative $\p_X f$  depends only on the   projection $pr (X) \in  \hat T M$.  

\begin{definition}\label{def:grad} A   section $\xi \mapsto \nabla _{\hat \g}  f(\xi) \in \hat  T^{\hat \g}_\xi M$  will be called the {\it  generalized  Fisher gradient} of a visible function  $f$, if for
all $X\in  T_\xi M$ we have
$$ df(X) =\hat \g  (pr (X), \nabla _{\hat \g} f).$$
\end{definition}

Clearly, if the generalized Fisher gradient  $\nabla_{\hat \g} f$ exists then it  is unique, and by the Riesz representation theorem the   generalized  Fisher gradient  of a  visible function $f$ exists iff for all $\xi \in M$ the linear functional $df_{\xi}$ is bounded  w.r.t.  the  reduced Fisher metric. As in \cite{Le2016} we denote   
 $$\Ll^k_1(\Om): = \{(f, \mu)|\, \mu \in \Mm(\Om) \text{ and } f \in L^k(\Om, \mu)\}.$$

 For a map $\pb: P \to \Mm(\Om)$ we denote by
$\pb^*(\Ll^k_1(\Om))$  the  pull-back ``fibration" (also called the fiber product)  $P \times _{\Mm (\Om)} \Ll^k_1(\Om)$.

\begin{definition}\label{def:pregrad} 
 Let $h$ be a  visible   function on $M$.    A  section
$$M \to \pb^*(\Ll^2_1(\Om)), \, \xi \mapsto \nabla h_\xi \in L^2(\Om, \pb(\xi)), $$
 is called {\it a pre-gradient of $h$}, if   for all $\xi \in M$ and  $X \in T_\xi M$  we have  
$$ dh (X) = \E_{\pb(\xi)}((\p_X \log  \pb)  \cdot \nabla h_\xi) .$$
\end{definition}

By definition,  a pre-gradient  of a  visible   function, if it exists,  is only  determined up to   a term  that is   $L^2$-orthogonal  to the   image  
$d \pb (T_\xi P) \subset L^2 (\Om, \pb (\xi))$.

\begin{lemma}\label{lem:exi1}  The existence  of  a pre-gradient  of  a   visible function $h$ implies  the existence of  the generalized  gradient  of $h$.
\end{lemma}

\begin{proof}   For  any $ \xi\in P$   the  map 
\begin{equation}
e: \hat T_\xi P \to L^2 (\Om, \pb(\xi)), \, X \mapsto \p _X \log \pb (\cdot;\xi),\label{eq:embe}
\end{equation}
  is  an embedding.   Here  $\p _X \log \pb$  denotes the value $\p _{\tilde X} \log \pb$ for some (and  hence any) $\tilde  X \in pr ^{-1} (X) \in T_\xi P$.
    The embedding  $e$ is
an isometric embedding   w.r.t. the Fisher  metric $\hat \g$ on  $\hat T_\xi P$ and
the $L^2$-metric in $L^2 (\Om, \pb(\xi))$, according to the definition
of the  (reduced) Fisher metric.  
The isometric embedding $e$  extends to an isometric embedding, also denoted by $e$,  of the  closure $\hat T_\xi  ^{\hat \g} P$  by setting  for  any limiting sequence  $\{ v_k \in  \hat T_\xi  P\}$
$$e(\lim_{k\to \infty}  v_k) : = \lim_{k \to \infty} e(v_k).$$ 
Now assume that $\nabla  f$  is  a pre-gradient of $f$.
Denote by  $\Pi$ the orthogonal  projection of  $L^2 (\Om, \pb (\xi))$ onto the closed   subspace  $e(\hat T^{\hat\g}_\xi M)$.
 Then
$$ \nabla _{\hat \g} f =  e^{-1}  ( \Pi  (\nabla  f)).$$
This completes the proof  of Lemma \ref{lem:exi1}.
\end{proof}

\begin{proposition}\label{prop:cov} 1. Let $(M, \Om,  \pb)$  be a  2-integrable  parametrized  measure model   and $f \in L^2_M (\Om)$  a regular  function.
Then  the section of the pullback fibration $\pb^*(\Ll^2_1(\Om))$ defined by $ \xi \mapsto  f \in L^2 (\Om, \pb(\xi))$ is a pre-gradient of the  visible function $\E_{\pb (\xi)}(f)$.

2. Let $(M, \Om,  \pb)$  be a  2-integrable   statistical model   and $f \in L^2_M(\Om)$ a regular function.  
Then  the section of the pullback fibration $\pb^*(\Ll^2_1(\Om))$ defined by $ \xi \mapsto f - \mathbb \E_{\pb(\xi)}(f) \in L^2 (\Om, \pb(\xi))$ is a pre-gradient of the  visible function $E_{\pb (\xi)}(f)$.
\end{proposition}

\begin{proof} 
Let $X \in T_\xi M$.  Using  Proposition \ref{prop:change-diff-int}  we obtain
\begin{equation}
\p_X \E_{\pb (\xi)}(f)(\xi) =  \int_\Om  f  \cdot \p _X \log  \pb (\xi)\, \pb (\xi)\label{eq:coma}
\end{equation}
we obtain the first assertion of Proposition \ref{prop:cov}.

To prove the second assertion we use the following   identity, which is a consequence of 
  (\ref{eq:coma})
\begin{equation}
\int_\Om \p_X  \log  \pb (x;\xi)\, d\pb(x;\xi) = 0.\label{eq:id1}
\end{equation}
 Multiplying (\ref{eq:id1}) with  $(- \mathbb \E_{p(\xi)}(f))$,   and plugging it  into (\ref{eq:coma}),
 we  obtain
$$\p_X \E_{\pb (\xi)}(f) =\int_{\Om} (f(x) -\mathbb E_{p(\xi)}(f) )\cdot \p_X \log  \pb (x;\xi)\, d\pb(x;\xi),$$
which implies  the second assertion of  Proposition \ref{prop:cov}.
\end{proof}

\subsection{Application to the case of finite sample  spaces}\label{subs:finite}  
 Let $\Om_n: = \{ \om_1, \cdots, \om_n\}$ be a finite  sample space  of $n$ elementary events. 
In this subsection we  apply the  formalism of visible functions  and their    (pre)-gradients to
  compute the Fisher metric, its inverse  and    the  Fisher gradient  of a function on  $\Mm_+ (\Om_n)$  and  its restriction to $\Pp_+ (\Om_n)$. Since the parametrization  of $\Mm_+ (\Om_n)$   is natural,  we have $\hat \g = \g$.
	
 Denote by  $L(\Ss(\Om_n), \R)$   the space  of $\R$-valued linear functions on $\Ss(\Om_n)$.
As in  Example \ref{ex:push}, we consider the following  canonical  linear  map 
$$E: \R^{\Om_n}  \to L(\Ss(\Om_n), \R)   $$ 
$$\la E(f) , \mu\ra  := \E_\mu (f) =  \int_\Om  f d\mu = \sum _{i=1}^n  f(\om_i) \mu(\om_i).$$
Here $\E_\mu$ 
stands for  the expectation  w.r.t. to the  (signed) measure $\mu \in \Ss (\Om_n)$.  

\begin{proposition}\label{prop:inv}  1) For any  $f \in \R^{\Om_n}$  and   any  $\mu\in \Ss (\Om_n)$ we  have
$$d  E(f) _\mu= E(f) \in  L(\Ss(\Om_n), \R) = T_\mu^* \Ss(\Om_n) .$$

2)  For  any $\mu \in \Ss(\Om_n)$ the  space  $\{ d E(f)_\mu |  f \in \R^ {\Om_n}\}$   coincides  with   $T^*_\mu \Ss(\Om_n)$.

3) Denote by $\g$ the Fisher metric on  $\Mm_+ (\Om_n)$.   
Then  for any $f, g \in \R^{\Om_n}$ we have
 $$ \g ^{-1}_\mu  (dE(f), dE(f)) =  E_\mu ( f\cdot g). $$
\end{proposition}

\begin{proof}  1. The first assertion holds, because $E(f)$ is a linear functional on $\Ss (\Om_n)$.

2.  The  second assertion follows from  the first one, noting  that $\dim (E(\R^{\Om_n})) = n = \dim  \Ss (\Om_n)$.

3. Let  us  prove the last assertion. Assume that $\mu \in \Mm_+(\Om_n)$.   Then  there exists a  linear   isomorphism
\begin{equation}
\Lambda_\mu: T_\mu (\Ss(\Om_n)) \to L_2 (\Om, \mu), \;   X \mapsto  \p _X \log  \bar \mu, \label{eq:log1}
\end{equation}
where  $\mu = \bar \mu \cdot \mu_0$  for some  $\mu_0 \in \Mm_+  (\Om_n)$.   It is known  that  the  RHS  of (\ref{eq:log1}) does not depend
on the choice   of $\mu_0$ and by (\ref{eq:log}) we  also have  $\p _X \mu =  \p _X \log (\bar \mu) \cdot  \mu$.
Since  $\Lambda_\mu$ is an isomorphism, Proposition \ref{prop:cov}.1 yields immediately

$$\Lambda_\mu ( \nabla_{\g}  E(f)  _\mu) =  f\in L ^2 (\Om_n, \mu) . $$

Hence
\begin{eqnarray*}
 \g ^{-1}_\mu( d E(f),  d E(g)) = \g_\mu ( \nabla_{\g}  E(f), \nabla_{\g} E(g) )  \\
= \int _{\Om_n}  \Lambda _\mu (\nabla_{\g}E(f)) \cdot \Lambda _\mu  (\nabla_{\g}E(g))\, d\mu = E_\mu (f\cdot g).   
\end{eqnarray*}
This proves  the  third assertion immediately.
\end{proof}

For a constant $c \in \R$ denote by $c_{|  \Om_n}$ the constant  function on  $\Om_n$  taking  value $c$.

\begin{proposition}\label{prop:prop} The    induced (inverse)  Fisher metric $\g ^{-1}$  on $T^*\Pp_+ (\Om_n)$ has  the following   form
$$\g^{-1} (d E(f), dE(g))  =  E_\mu  [( f -E_\mu (f)_{| \Om_n}) \cdot ( g -E_\mu (g)_{|\Om_n})]  .$$
\end{proposition}
\begin{proof}   Note  that  for any constant $c$  the restriction  of  $dE (c_{|\Om_n})$ to  $T^*_\mu \Pp_+$ vanishes  and hence
$$\Lambda_\mu (T_\mu \Pp_+ (\Om_n)) = \{  g \in L^2 (\Om_n, \mu)|\,  E_\mu (g) = 0\},$$
 we obtain easily
\begin{equation}
\Lambda _\mu  (\nabla_{\g}E(f))  = f - E_\mu (f)_{|\Om_n}  \in  T^*_\mu (\Pp _+ (\Om_n)) = j ^* (T^*_\mu  (\Pp (\Om_n)).\label{eq:proj}
\end{equation}
This  proves  Proposition  \ref{prop:prop}.
\end{proof}

\begin{remark}\label{rem:explicit} Let $\delta_i$  denote the Dirac  function  on $\Om_n:  \delta_i (\om_j) = \delta ^i_j$.
The first assertion of Proposition \ref{prop:inv}   implies that  $\{E(\delta_i)|\, i= \overline{1,n}\}$   form a  basis  of the  $C^\infty$-algebra of smooth functions on $\Ss(\Om_n)$ 
(and resp. on  the open set $\Mm_+ (\Om_n)$  of $\Ss(\Om_n)$). In other words  we can take $E(\delta_i)$  to be coordinates of  $\Mm_+(\Om_n)$.
Writing  $\mu = \sum \mu_i \hat \delta_i$, where $\hat \delta_i$ denotes  the  Dirac measure concentrated  at $\om_i$, we have
$$ E(\delta_i)  (\mu) = \mu_i.$$
So we can identify  $E(\delta_i)$ with $\mu_i$.
Proposition  \ref{prop:inv} implies that  
\begin{equation}
\g (\mu) = \sum_{i=1}^n \frac{1}{\mu_i} d\mu_i ^2 .\label{eq:fisher1}
\end{equation}
By definition, the Fisher metric  on $\Pp_+ (\Om_n)$ equals the restriction of  the Fisher metric  on $\Mm_+ (\Om_n)$  to $\Pp_+ (\Om_n)$.
\end{remark}

\begin{proposition}\label{prop:fgradn}  Let $\tilde f$ be a function on $\Mm_+(\Om_n)$. Then
\begin{equation}
\nabla_\g \tilde f (\mu) = \sum_i \mu_i {\p \tilde f \over \p  \mu_i}  \p \mu_i.\label{eq:fgradnm}
\end{equation}
Let $f$  be the restriction of $\tilde f$ to  $\Pp_+ (\Om_n)$. Then
\begin{equation}
\nabla_\g  f (\mu) = \sum_i \mu_i (\frac{\p \tilde f}{\p \mu_i} - \lambda) \p \mu_i, \label{eq:fgradnp}
\end{equation}
where
$$\lambda =  \sum_i \mu_i {\p \tilde f\over \p \mu_i}.$$
\end{proposition}
\begin{proof}
1. The first  equation follows immediately from (\ref{eq:fisher1}).

2. Note  that the Fisher gradient  of the restriction  $f$ of a function  $\tilde f$  to $\Pp_+ (\Om_n)$ is the  projection   of the gradient of $\tilde  f$:
$$\nabla _\g f (\mu)  = Pr (\nabla _\g  \tilde f ), $$
where $Pr$ denotes  the  (Fisher) orthogonal projection  on the   tangent  space of  $\Pp_+ (\Om_n)$.
Since  the function $ w(\mu): = \sum _i  \mu_i$ is   equal to 1  on $\Pp_+ (\Om)$, its  Fisher  gradient   $\nabla_\g w = \sum_i \mu_i \p\mu_i $ is orthogonal
to the tangent space  $T_\mu \Pp_+ (\Om_n)$.  Thus  the  Fisher  gradient  of $f$ on $\Pp_+ (\Om_n)$  has the form (\ref{eq:fgradnp}),
where
$$ \lambda = \frac{\g(\nabla_\g  \tilde f , \nabla_\g \tilde  f)}{\g (\nabla_\g w, \nabla _g  w)} = \frac{\g ^{-1} (d\tilde f,\sum_i  d\mu_i)}{\g^{-1} (\sum_i  d\mu_i,\sum_i d\mu_i)}= \sum_i \mu_i {\p \tilde f\over \p \mu_i}.$$
The last equality  follows from Proposition \ref{prop:inv}.3, taking into account $\sum_i \mu_i = 1$.

This completes the proof of Proposition \ref{prop:fgradn}.
\end{proof}
\begin{remark}
	\label{rem:grad} Proposition \ref{prop:fgradn}  shows that  the Fisher  gradient
 of a $C^1$- function $f$ on $\Mm_+ (\Om_n)$ extends smoothly to  the    whole space
 $\Mm (\Om_m)$,  if   $f$  is   the     restriction  of a $C^1$-function $\tilde  f$ on $\Mm(\Om_n)$. 
   Here the smooth  structure  of  the manifold with corner  $\Mm(\Om_n)$ is defined by the natural  inclusion $\Mm(\Om) \INTO  \Ss(\Om) = \R ^n$. 
   This  continuity holds  because  the  inverse $\g^{-1}$ of the Fisher metric $\g$   on $\Mm(\Om)$
 is  a continuous   2-vector  on $\Mm(\Om)$.  This observation  suggests that  a certain blow-up  type of  the Fisher metric  should also   be considered  when we generalize   the classical   Cram\'er-Rao inequality.   We refer  the reader to 
 the  next   part  of our  paper  for  details \cite{JLS2017b}.
\end{remark}


\section{Cram\'er-Rao  inequality on singular  statistical models}\label{sec:CR}

In this section we assume that $(P, \Om, \pb)$ is a 2-integrable statistical model, 
 $V$  a topological vector space 
  and  $\hat \sigma  \in L^2_{\varphi} (P, \Om)$  an estimator for a  $V$-valued   function $\varphi$ on $P$.
	We prove   a general  Cram\'er-Rao inequality (Theorem \ref{thm:CR}) for $\varphi$-regular estimators $\hat \sigma$,  using the  notion of essential tangent  space
	and  reduced Fisher  metric and results in the previous sections.  At the end of the section,  we   derive  from  Theorem \ref{thm:CR}   classical  Cram\'er-Rao inequalities,  compare our results  with other generalizations of the Cram\'er-Rao inequality  and  summarize  our main contributions in this paper.

\subsection{Bias, mean square error and  variance of  an estimator}\label{subs:bias}
In this subsection we  recall the notion of the  bias,  the mean square error   and   the variance of  an estimator   and their relation,  which are generalized immediately
in  our  proposed general setting.
\begin{definition}
The  difference 
\begin{equation}
b^{\varphi}_ {\hat \sigma} :=\varphi_{\hat \sigma} -\varphi \in V^P \label{eq:bias}
\end{equation}
 will be called the {\it  bias of the estimator $\hat \sigma$ w.r.t. the map $\varphi$}. 
\end{definition}

\begin{definition}\label{def:unbi} Given  an estimator 
$\hat \sigma \in L^2_{\varphi} (P, \Om)$ the estimator $\hat \sigma$  will be called {\it $\varphi$-unbiased}, if 
$\varphi_{\hat \sigma} = \varphi$, equivalently, $b^\varphi_{\hat \sigma} = 0$.  
\end{definition}

 Given $\hat \sigma \in L^2_{\varphi} (P, \Om)$, we define  the {\it  $\varphi$-mean square error}    of  an  estimator $\hat \sigma: \Om \to P$  to be the  quadratic  
form   $MSE_{\pb (\xi)}  ^{\varphi}[\hat\sigma]$ on $V^*$   such that for  each $l, k \in V^*$ we have

\begin{equation}
MSE_{\pb (\xi)} ^{\varphi}[\hat\sigma] (l, k) : = \E_{\pb (\xi)}[(\varphi ^l \circ \hat\sigma - \varphi^l \circ  \pb (\xi))\cdot ( \varphi ^{k}\circ \hat\sigma - \varphi^{k} \circ  \pb (\xi))].
\label{eq:mse}
\end{equation}

We also  define the {\it   variance of $\hat\sigma$ w.r.t. $\varphi$}
    to be  the quadratic  form $V_{\pb (\xi)} ^{\varphi}[  \hat\sigma]$ on $V^*$ such that  for all $l, k \in V^*$ we have
\begin{equation}
V_{\pb (\xi)} ^{\varphi}[ \hat\sigma]( l, k) : = \E_{\pb (\xi)}[(\varphi ^l\circ \hat\sigma - \E_{\pb (\xi)} (\varphi^l \circ \hat\sigma))\cdot (\varphi ^{k}\circ \hat\sigma - \E_{\pb (\xi)} (\varphi^{k} \circ \hat\sigma))].\label{eq:var}
\end{equation}
The  RHSs of  (\ref{eq:mse}) and  (\ref{eq:var}) are well-defined, since $\hat\sigma \in L^2_{\varphi}(P, \Om)$.

We shall also use  the following relation
\begin{equation}
MSE_{\pb (\xi)} ^{\varphi}[\hat \sigma ](l,k)  = V_{\pb (\xi)} ^{\varphi}[  \hat\sigma]( l, k)  +  \la b^{\varphi}_{ \hat \sigma}(\xi), l\ra\cdot \la b^{\varphi}_{ \hat \sigma}(\xi), k\ra \label{eq:msevb}
\end{equation}
for all $\xi \in P$  and all $l,k \in V^*$.
Since    for a given $\xi \in P$ the LHS  and RHS of  (\ref{eq:msevb})  are   symmetric bilinear forms  on  $V^*$, it suffices   to prove  (\ref{eq:msevb}) in the case $k = l$. 
We write
$$\varphi ^l\circ \hat\sigma -  \varphi^l \circ  \pb (\xi) = (\varphi ^l\circ \hat\sigma- \E_{\pb (\xi)} (\varphi^l \circ \hat\sigma))  + (\E_{\pb (\xi)} (\varphi^l \circ \hat\sigma) -  \varphi^l \circ  \pb (\xi) )$$
$$ =(\varphi ^l\circ \hat\sigma -\E_{\pb (\xi)} (\varphi^l \circ \hat\sigma)) + \la b^{\varphi}_{\hat \sigma}  (\xi), l\ra.$$
Taking into account that  $\pb (\xi)$ is a  probability measure,  we obtain
\begin{equation}
MSE_{\pb (\xi)} ^{\varphi}[\hat \sigma ](l,l) = V_{\pb (\xi)} ^{\varphi}[\hat \sigma](l, l)  + \la b^{\varphi}_{\hat \sigma}  (\xi), l\ra^2  + 2\int_\Om (\varphi ^l\circ \hat\sigma- \E_{\pb (\xi)} (\varphi^l \circ \hat\sigma))\cdot \la b^{\varphi}_{\hat \sigma}  (\xi), l\ra  d\pb (\xi).\label{eq:msevb1}
\end{equation}
Since $\la b^{\varphi}_{\hat \sigma}  (\xi), l\ra$ does not depend   on $x$, it can be taken out of the integral, and therefore    the last term in the RHS of (\ref{eq:msevb1})  vanishes.  As we have noted  this      proves (\ref{eq:msevb}).

\subsection{A general Cram\'er-Rao inequality}\label{subs:cr}

\begin{proposition}\label{prop:var} Let $(P, \Om,  \pb)$   
be a   2-integrable   statistical model, $\varphi$ - a $V$-valued function on $P$ and  $\hat \sigma \in L^2_{\varphi} (P, \Om)$ - a $\varphi$-regular  estimator.  Then for any $l \in V^*$   and any  $\xi \in P$ we have
$$V_{\pb(\xi)}^{\varphi}[ \hat \sigma](l, l)  : =  \E_{\pb (\xi)}(\varphi ^l\circ \hat\sigma - \E_{\pb (\xi)} (\varphi^l \circ \hat\sigma))^2 \ge \|d\varphi^l_{\hat \sigma}\|^2_{\hat \g^{-1}}(\xi).$$
\end{proposition}

\begin{proof} Recall  that $e: \hat  T_\xi  ^{\hat  \g}  P  \to  L^2 (\Om, \pb (\xi))$  is an isometric  embedding.
Since $e(\hat T_\xi^{\hat \g} P)$ is   a closed subspace  in $L^2 (\Om, \pb(\xi))$,  we have the orthogonal decomposition
\begin{equation}
 L^2(\Om, \pb(\xi)) = e(\hat T_\xi^{\hat\g} P) \oplus   e(\hat T_\xi^{\g} P) ^\perp.\label{eq:fdecom}
\end{equation}
Denote  by $\Pi_{e (\hat T_\xi^{\hat \g} P)}$ the  orthogonal  projection $ L^2 (\Om, \pb(\xi))$ to $e(\hat T_\xi ^{\hat \g} P, \hat \g)$ according to the  above  decomposition.

By Proposition \ref{prop:cov}.2,   $\varphi^l\circ \hat \sigma - \mathbb \E_{\pb(\xi)}(\varphi^l\circ \hat \sigma)$ is a pre-gradient  of $\varphi^l_{\hat \sigma}$. Hence
\begin{equation}
\Pi_{e (\hat T_\xi^{\hat \g} P)} ({\varphi^l\circ \hat \sigma} - \mathbb \E_{\pb(\xi)}({\varphi^l\circ \hat \sigma})) =  e(\grad _{\hat  \g} \varphi^l_{\hat \sigma}),\label{eq:proj2}
\end{equation}
for $\varphi \in L^2_{\hat\sigma } (P, V)$. 
Using  (\ref{eq:var}) and the decomposition (\ref{eq:fdecom}), we obtain
\begin{equation}
V_{\pb(\xi)}^{\varphi}[ \hat \sigma](l, l)  \ge \|\Pi_{e (\hat T_\xi P)} (\varphi^l\circ \hat \sigma - \mathbb \E_{\pb (\xi)}(\varphi^l\circ \hat \sigma))\|^2_{L^2(\Om, \pb (\xi))} .\label{eq:vpproj}
\end{equation}
Combining (\ref{eq:vpproj}) with (\ref{eq:proj2}), we   derive Proposition  \ref{prop:var} immediately from the following obvious identity (see Def. \ref{def:grad})
$$\|\nabla _{\hat  \g}\varphi^l_{ \hat \sigma}\|^2_{\hat \g}(\xi)  = \|d\varphi^l_{\hat \sigma}\|^2_{\hat \g^{-1}}(\xi).$$ 
\end{proof}

We regard  $\|d\varphi^l_{\hat \sigma}\|^2_{\hat \g ^{-1}} (\xi)$  as a quadratic form on $V^* $  and  denote the latter one  by $ (\hat\g^{\varphi}_{\hat \sigma})^{-1}(\xi)$,  i.e.
$$ (\hat \g^{\varphi}_{\hat \sigma})^{-1}(\xi) (l,k) : = \la d\varphi^l_{\hat \sigma},d\varphi^{k}_{\hat \sigma}\ra _{\hat\g ^{-1}} (\xi).$$

Thus we obtain from  Proposition  \ref{prop:var} the following.

\begin{theorem}\label{thm:CR}(Cram\'er-Rao inequality) Let $(P, \Om,  \pb)$   be a   2-integrable   statistical model, $\varphi$  a $V$-valued function on $P$ and  $\hat \sigma \in L^2_{\varphi} (P, \Om)$  a $\varphi$-regular  estimator.  Then the  
difference $V_{\pb(\xi)}^\varphi[\hat \sigma] -   (\hat \g^{\varphi}_{ \hat \sigma})^{-1} (\xi)$     is   a  positive semi-definite  quadratic form on $V^*$  for any $\xi \in P$. 
\end{theorem}

This is the general 
Cram\'er-Rao inequality.

\subsection{Classical Cram\'er-Rao inequalities}\label{subs:rcCR}
Our    generalization of the Cram\'er-Rao inequality { (Theorem \ref{thm:CR})}  does not  require   the nondegeneracy of the  (classical)  Fisher  metric  nor the  finite  dimensionality of  statistical models,  nor positivity  of the density functions of   statistical model. When we make such additional assumptions, we regain the various versions of the inequality known in the literature. We shall list  some important examples. After the initial work of  Rao and Cram\'er  on  information lower bounds \cite{Cramer1946, Rao1945},  many versions  of  Cram\'er-Rao   inequalities have appeared in the literature
see e.g. \cite[p. 317]{Witting1985} and the remainder of this  paper,  and the most general among them, as far as we are aware,    is  in \cite{Borovkov1998}.

\

(A)  Assume that  $V$ is finite dimensional and $\varphi$ is a coordinate mapping.  Then $\hat \g = \g$  and $d\varphi^l = d\xi^l$,  and  with (\ref{eq:bias}), abbreviating $b^\varphi_{\hat \sigma}$ as $b$,  we  write 
\begin{equation}
  (\g^{\varphi}_{ \hat \sigma})^{-1} (\xi) (l,k) =    \la \sum_i(\frac{\p\xi ^l}{\p \xi ^i} + \frac{\p b ^l}{\p \xi^i})d\xi^i, \sum_{j}(\frac{\p \xi^{k}}{\p \xi^{j}} + \frac{\p b^{k}}{\p \xi^{j}}) d\xi^{j}\ra_{\g^{-1}}(\xi).\label{eq:finite1}
	\end{equation}
Let $D(\xi)$  be the linear  transformation of  $V$   whose  matrix coordinates are
$$D (\xi)^l_k : = \frac{\p b ^l}{\p \xi^k}.$$
With (\ref{eq:finite1}), the   Cram\'er-Rao inequality in  Theorem  \ref{thm:CR}  becomes
\begin{equation}
V_\xi [\hat \sigma]  \ge  (\E  + D (\xi) )  \g ^{-1} (\xi)  (\E + D (\xi)) ^T. \label{eq:borovkov}
\end{equation}
 The inequality  (\ref{eq:borovkov})  coincides with  the Cram\'er-Rao inequality in \cite[Theorem 1.A, p. 147]{Borovkov1998}.
The condition (R)  in \cite[p. 140, 147]{Borovkov1998}   for the validity  of  the Cram\'er-Rao inequality  is  essentially  equivalent  to 
the  2-integrability  of the (finite dimensional) statistical  model  with  positive regular density  function under consideration, more precisely  Borokov ignores/excludes  the points $x \in \Om$ where
the density function  vanishes  for computing the Fisher metric. Since  we do not assume the  existence  of  a positive  regular  density  function, our set-up  is more general than   that by Borokov. 
Borovkov  also uses the  $\varphi$-regularity assumption, written as $\E_\theta ((\theta ^*)^2) < c < \infty$  for $\theta \in \Theta$, see also \cite[Lemma 1, p. 141]{Borovkov1998}  for a more precise  formulation. 

\

(B)  Specializing further and assuming that    $V = \R$ and $\varphi $ is a coordinate mapping. Then  
\begin{equation}
\E +  D(\xi) =  1 + b_{\hat \sigma } '\label{eq:d=1}
\end{equation}
 where $b_{\hat\sigma}$  is short for $b^{\varphi}_{ \hat \sigma} $.
Using (\ref{eq:d=1})   and (\ref{eq:msevb}), we derive from   (\ref{eq:borovkov})
\begin{equation}
\mathbb \E_\xi (\hat \sigma -\xi) ^2 \ge {[1+ b_{\hat \sigma }' (\xi)]^2\over \g (\xi)} + b_{\hat \sigma} (\xi) ^2. \label{eq:crib4}
\end{equation}
(\ref{eq:crib4}) is identical  with the Cram\'er-Rao  inequality
with a bias term   in \cite[(11.290) p.396,(11.323) p.402]{CT2006}.

\

(C)  Assume that  $V$ is   finite dimensional,  $\varphi$ is a coordinate mapping  and $\hat\sigma$ is $\varphi$-unbiased.
Then   the  terms involving $b_{\hat\sigma}$ vanish, and the Cram\'er-Rao  inequality  in  Theorem \ref{thm:CR}  becomes the well-known  Cram\'er-Rao inequality for an unbiased estimator (see  e.g. \cite[Theorem 2.2, p. 32]{AN2000})
\begin{equation}
 V_\xi [\hat\sigma] \ge \g ^{-1} (\xi). \nonumber
\end{equation}
 
\

(D)  In \cite[Chapter 5]{BKRW1998} Bickel-Klaassen-Ritov-Wellner  consider  efficient  estimations for infinite  dimensional statistical models. They  define the inverse  information  covariance  function  by looking 
at a variation $\p_V \pb$   in the Hilbert   space $L^2(\Om, \pb(x))$, which  is similar  to our idea  in the present paper. They did  not derive  an analogue of the  Cram\'er-Rao inequality. 
 They are mainly interested in the asymptotic  behavior  of  estimators.

\subsection{Janssen's nonparametric  Cram\'er-Rao inequality}\label{subs:jan}
In this Subsection  we compare  our parametric  Cram\'er-Rao  inequality with  Janssen's   nonparametric  Cram\'er-Rao 
inequality \cite{Janssen2003}, which, as far as  we aware of (Subsection  \ref{subs:other}),  is  the version closest  to our work.

The nonparametric setting of Janssen's   work   follows, in particular,   the line of Bickel et al. \cite{BKRW1998}.
Janssen considers a  general measurable space $\Om$  and    a subset $P \subset \Pp(\Om)$ of probability measures  for which he defined  the notion   of a tangent space  using the same method we  define  the tangent  space  for the subset   $\Mm^r(\Om) \subset \Ss^r(\Om)$ in \cite{AJLS2016b}. First  for  $\xi \in \Pp(\Om)$  Janssen  calls  elements of  the  set
$$L_2 ^ {(0)} (\xi) : = \{ g \in L_2 (\Om, \xi)|\,  \int_\Om g d\xi = 0\}$$
{\it tangents  at $\xi$}.
In our language  $T_\xi (\Pp^{1/2} (\Om)) = L_2 ^{(0)}(\xi)\cdot \xi ^{1/2}.$

Janssen calls a curve  $\gamma:  I\to \Pp (\Om)$ {\it $L_2$-differentiable}  at $t = 0 \in I$,  if  there
exists  a tangent   $g \in L_2 ^ {(0)}(\gamma(0))  $ such that  for all sequences $t_n \to 0$  and each finite  dominating measure $\mu$ of $\{  \gamma (t_n) |\,n \in \N \} \cup  \gamma(0)$   we have \cite[(2)]{Janssen2003}
$$\frac{2}{t_n}[ ( \frac{d\gamma (t_n)}{d\mu})^{1/2}- (\frac{d\gamma(0)}{d\mu})^{1/2})] - g( \frac{d\gamma (0)}{d\mu})^{1/2}   \to 0 \in L_2 (\Om, \mu)$$
as $n  \to \infty$.

In our language, using Proposition \ref{prop:mfld-dom}, a curve  $\gamma:  I\to \Pp (\Om)$ is $L_2$-differentiable  iff    the composition  $\pi ^{1/2} \circ \gamma: I \to \Pp^{1/2} (\Om)$ is differentiable. By Theorem \ref{thm:k-integr} the  curve $\gamma(t)$ is 2-integrable.
Moreover  Theorem \ref{thm:k-integr}  asserts that  the $L_2$-differentiability
is equivalent to  the seemingly weaker condition  of   weak continuity  of the Fisher metric.  Thus Theorem \ref{thm:k-integr}  also clarifies the general set-up of Janssen's work.

 Janssen's  statistical functional $\kappa: P \to \R$ is
a (particular) version of   our feature   function $\varphi: P \to V$. (We  shall  discuss  Janssen's general statistical functional $\kappa: P \to  M$ below.) In our  notations, $P$ stands for a parameter  space, and therefore  its image
$\pb (P)$  corresponds  to Janssen's subset $P$ of  probability measures  on
$\Om$.  Thus the composition $\varphi: =\kappa \circ \pb$ is a feature  function.
The difference   is that Janssen wants to estimate  a  probability  measure $\xi \in P$ and we want to estimate {\it  the parameter} of  a  probability measure $\xi$.

In our  work, an estimator   is   a map $\hat \sigma: \Om \to  P$. In Janssen's work
an estimator is a map $\hat \sigma : \Om \to  \R$. So  the composition $\varphi \circ \hat{\sigma }$ in our work corresponds  to   an estimator  in Janssen's work.
Note that  taking    the composition $\varphi \circ \hat \sigma$  is specially important for    estimators on a singular  statistical model, since  only in this form, the function  $\E _{\pb (\xi)}(\varphi^l \circ \hat \sigma)$ is visible, and therefore      in our setting, we can use   the  reduced  Fisher metric,  which is not present  (and not necessary) in nonparametric setting.  This is  the  most  important difference   between our      work and Janssen's work.

In Janssen's  work   the  rule of differentiation under the integral  sign 
\cite[Lemma 1, p. 349]{Janssen2003}  is  a partial case  of  our Proposition \ref{prop:change-diff-int}, namely  for $k = k' = 2$,  and it was known  before
Janssen's work, see  loc. cit.  Once we have this rule,   the    Cram\'er-Rao  
inequality  for  an estimator  $T$  is derived by Janssen  in a similar way  as in   the proof of Proposition  \ref{prop:var}.   Janssen also considers   the general  nonparametric Cram\'er-Rao  inequality  for  an arbitrary   statistical functional $\kappa: P \to M$,
which is expressed  in terms  of a non-negative quadratic form  on a linear space $W$
 of functions on $M$ (so his setting formally    is slightly   larger than ours, where we assume $M$ is a topological vector space  $V$  (and $ W = V^*$) but  essentially equivalent,
 since    the  Cram\'er-Rao inequality depends    on   {\it   the  linear space} $W$ (resp. $V^*$).

\subsection{Comparing with other  generalizations  of   Cram\'er-Rao    inequality}\label{subs:other}
In this paper, using  and  developing our theory   for (possibly infinite dimensional)  parametric   measure models  in \cite{AJLS2015, AJLS2016b},   we are concerned  with  a generalization  of the Cram\'er-Rao inequality  in a parametric setting  where  the Fisher metric may  be degenerate,
the  statistical measure  model under consideration may be infinite dimensional and does not need to  consist  of dominated measures, and moreover,   estimators do not need   to be  unbiased.

We    would like to stress  that  there are many    different  generalizations of the original  Cram\'er-Rao  inequality   \cite{Rao1945, Cramer1946}.
  We  searched in the database of 
Math.Sci.Net  under the key word  ``Cramer-Rao"    in  the title     of  the   papers  reviewed or indexed by  Mat. Sci. Net. 
The query  returns  209  matched items  on  May 08,   2017.
A large    amount of  papers    from the  209 matched items  are devoted to 
applications    and refinements of the Cram\'er-Rao  in special situations.
  
Many generalizations  from the 209 items    are     particular cases  of   our generalization  in the present paper.    Other generalizations  concern  Cram\'er-Rao  type inequality  w.r.t. to a generalized  Fisher metric
(e.g.   in  a  quantum information setting or   the   $q$-Fisher-metric, or    the Fisher metric    derived from other  convex functions),   or   w.r.t.  the generalized covariance  of  an estimator.  There are  also a few  papers  discussing     generalizations of  the Cram\'er-Rao inequality in the presence of a singular  Fisher metric, see  the  item (4) below. We  refer  the reader to  the    paper by Cianchi-Lutwak-Yang-Zhang \cite{CLYZ2014}  and the references therein   as  the  best  recent short   survey   on  generalizations of  Cram\'er-Rao inequalities. Note that the paper by Cianchi-Lutwak-Yang-Zhang is  exclusively concerned with  {\it regular 1-dimensional}  statistical models. Thought the   Cram\'er-Rao inequality    generally  boils down  to an   equality   depending on   a given 
tangent  vector, a   good formulation  for multi-dimensional (possibly infinite dimensional)  statistical models  is  important; in fact,   we don't know  of any 
example  of  an infinite   dimensional  exponential  model  that admits an efficient
unbiased  estimator \cite{JLS2017b, SFKGH2013}.  The regularity  assumption   has been  discussed   in   \cite{AJLS2016b}, see also  Remark \ref{rem:reg}.
Among   results that have not been   discussed  in \cite{CLYZ2014} we  would   like to mention    the  original paper   by  Espinasse \cite{Espinasse2012}, whose generalization   drops the smoothness   assumption of  the statistical model. 
  
 Our results   are most  closely related to Janssen's  nonparametric  Cram\'er-Rao  inequality, which we reviewed above. 
 
 \subsection{Conclusion}\label{subs:conclusion}
 
 To sum up, the most important   contributions  in our paper  are the   following.
 \begin{enumerate}
 \item  Our Cram\'er-Rao lower bound  is defined in the most general terms  based on  
 	our theory  of parametrized measure models developed in \cite{AJLS2015, AJLS2016b}  that encompasses 
 	many partial cases which use more complicated terminology, e.g.   regarding   separate case of Riemannian submanifolds as in \cite{Boumal2013}.   We spell
 	out   properties  of estimators  and estimations  that   do not depend
 	on  the parametrization of  the statistical models under consideration. In particular,   our theory covers the intrinsic  Cram\'er-Rao lower bound 
 	 introduced by S.T. Smith in 2005 
 	\cite{Smith2005}  and developed later  in \cite{Boumal2013}. (The  intrinsic  estimate setting in \cite{Smith2005} is not  completely intrinsic: it is a {\it density estimation problem}  and therefore  depends on the choice  of a dominant  measure. Furthermore, Smith chooses a special  feature   function $\varphi$ using  the exponential map, which is not always globally defined.) 
 	
 	\item  The Fisher  metric in our setting  is defined  without any assumption on the existence 
 	of  a dominating  measure. The   closest  treatment  by Janssen \cite{Janssen2003}  is technically more complicated  and less complete than ours.
 	
 	\item  We clarified  the relation between      different  sufficient  regularity  conditions  (in our language - $2$-integrability) on  the   statistical models. 
 	
 	\item  We treat  the case of a singular  Fisher metric  using the reduced metric.  
 	 The classical remedy  for singularities  of the Fisher metric  is to use the Moore-Penrose pseudoinverse, hereafter referred to as
 	the pseudoinverse, of the Fisher  information matrix, see e.g. \cite{Boumal2013, BHE2009}. Geometrically our  approach is simpler  and geometrically clearer,  which  is particularly important  for  the consideration of the  case  when 
 	our Cram\'er-Rao inequality is optimal, see   our subsequent paper \cite{JLS2017b}.   Our   formulation has an advantage  over the  use  of the Moore-Penrose  pseudoinverse,  since  the later one  is defined 
	{\it with the  help
	of  another    non-degenerate metric}.  It is not hard  to see  that the  Moore-Penrose pseudoinverse  is   equal  to the  inverse   of the  reduced  Fisher  metric.
 \end{enumerate}

Finally we     remark that our theory   can be  coherently and straightforward  extended to  other  natural statistical models with   different types  of singularities,  including  important   compactifications of open  statistical models, e.g.   the  statistical model $\Pp(\Om_n)$ of all nonnegative  probability measures  on a finite  set $\Om_n$ \cite{JLS2017b, JLS2017c}.

\section*{Acknowledgement}
A part of this  paper has been  discussed  during  extended  visits  by HVL and LS  to the Max Planck Institute  for Mathematics  in the  Sciences, and they thank the institute for its hospitality and providing excellent working conditions.


\begin{thebibliography}{99999}
\bibitem[AJLS2015]{AJLS2015}{\sc  N. Ay,  J.   Jost, H. V. L\^e,  and  L. Schwachh\"ofer},  Information geometry and sufficient statistics,  Probability  Theory and related Fields,  162 (2015), 327-364, arXiv:1207.6736.
\bibitem[AJLS2016]{AJLS2016}{\sc N. Ay, J.  Jost, H. V. L\^e,  and  L.  Schwachh\"ofer},   Information  geometry,   Springer  2017 (in press).
\bibitem[AJLS2016b]{AJLS2016b}{\sc N. Ay, J.  Jost, H. V. L\^e,  and  L.  Schwachh\"ofer}, Parametrized  measure models, (accepted for  Bernoulli Journal),  arXiv:1510.07305.
\bibitem[Amari1987]{Amari1987}{\sc S.  Amari},  Differential Geometrical Theory of Statistics, in: Differential geometry in statistical inference, Institute of Mathematical Statistics, Lecture Note-Monograph Series, Volume 10,  California, 1987.
\bibitem[Amari2016]{Amari2016}{\sc S. Amari}, Information Geometry and Its Applications, Springer, Applied Mathematical Sciences, Volume 194,  2016.
\bibitem[AN2000]{AN2000}{\sc S. Amari,  H.  Nagaoka},  Methods of information geometry, Translations of mathematical monographs; v. 191, American Mathematical Society,  Providence, RI; Oxford University Press, Oxford, 2000.
\bibitem[BHE2009]{BHE2009}{\sc Z. Ben-Haim, Y.C. Eldar}, On the Constrained Cram\'er- Rao Bound With a Singular Fisher Information Matrix, 
IEEE  Signal Processing Letter,   16(2009), 453-457.
\bibitem[Bercher2012]{Bercher2012}{\sc J.-F. Bercher}, On generalized Cram\'er-Rao inequalities, generalized Fisher information and characterizations of
generalized $q$-Gaussian distributions, Journal of Physics A: Mathematical and Theoretical, 45(2012), 255-303. 
\bibitem[BKRW1998]{BKRW1998}{\sc P. Bickel, C. A. J. Klaassen, Y. Ritov, J. A. Wellner}, Efficient and  Adaptive  Estimation for Semiparametric Models,  Springer, 1998.
\bibitem[Borovkov1998]{Borovkov1998} {\sc  A. A. Borovkov}, Mathematical statistics, Gordon and Breach Science Publishers, 1998.
\bibitem[Boumal2013]{Boumal2013}{\sc N. Boumal},  On intrinsic Cram\'er-Rao bounds for Riemannian submanifolds and quotient manifolds, IEEE Trans. Signal Process. 61 (2013), no. 7, 1809-1821.
\bibitem[Chentsov1982]{Chentsov1982}{\sc N.\ Chentsov}, Statistical decision rules
and optimal inference,  Moscow,  Nauka,  1972 (in Russian),   English translation  in:  Translation of Math.\ Monograph 53, American Mathematical Society, Providence, RI (1982).
\bibitem[CLYZ2014]{CLYZ2014}{\sc A. Cianchi, E. Lutwak,  D. Yang, G. Zhang},  A unified approach to Cram\'er-Rao inequalities. IEEE Trans. Inform. Theory 60 (2014), no. 1, 643-650.
\bibitem[CT2006]{CT2006}{\sc  T. M. Cover  and  J. A. Thomas}, Elements of Information theory, Wiley and Sons,  second edition, 2006.
\bibitem[Cramer1946]{Cramer1946}{\sc  H. Cram\'er},  Mathematical Methods of Statistics,  Princeton Univ. Press, Princeton, 1946.
\bibitem[Espinasse2012]{Espinasse2012}{\sc T. Espinasse, P. Rochet},  A Cram\'er-Rao inequality for non-differentiable models. C. R. Math. Acad. Sci. Paris 350 (2012), no. 13-14, 711-715.
\bibitem[Janssen2003]{Janssen2003}{\sc A. Jansson}, A nonparametric Cram\'er-Rao
inequality, Statistics \&  Probability letters, 64(2003), 347-358.
\bibitem[JLS2017b]{JLS2017b}{\sc J. Jost,  H. V. L\^e and L. Schwachh\"ofer},  The Cram\'er-Rao inequality on singular statistical models II (in preparation).
\bibitem[JLS2017c]{JLS2017c}{\sc J. Jost,  H. V. L\^e and L. Schwachh\"ofer}, In preparation.
\bibitem[Le2016]{Le2016}{\sc H.V. L\^e}, The uniqueness of the Fisher metric as information metric, AISM, 69 (2017),    arXiv:math/1306.1465.
\bibitem[Rao1945]{Rao1945}{\sc C. R. Rao}, Information and the accuracy
attainable in the estimation of statistical parameters, Bulletin of
the Calcutta Mathematical Society 37  (1945), 81-89.
\bibitem[SFKGH2013]{SFKGH2013}{\sc B. K. Sriperumbudur, K. Fukumizu, R. Kumar, A. Gretton and A. Hyv\"arinen},  Density Estimation in Infinite Dimensional Exponential Families, arXiv:1312.3516.
\bibitem[Smith2005]{Smith2005}{\sc S.T. Smith},  Covariance, subspace, and intrinsic Cram\'er-Rao bounds, IEEE-Transactions on Signal
Processing, 53(5):1610-1629, 2005.
\bibitem[Vapnik1999]{Vapnik1999}{\sc V. N. Vapnik}, The nature of statistical learning theory, Springer, 1999.
\bibitem[Wasserman2006]{Wasserman2006}{\sc L. Wasserman}, All of Nonparametric Statistics, Springer, 2006.
\bibitem[WMS2008]{WMS2008}{\sc D. Wackerly, W. Mendenhall,  R. L. Scheaffer},  Mathematical Statistics with Applications, Thomson Higher Education, Belmont, CA, USA  (2008).
\bibitem[Watanabe2007]{Watanabe2007}{\sc S. Watanabe},  Almost all learning machines are singular, In the Proceedings of the IEEE Int. Conf.
FOCI, pages 383-388, 2007.
\bibitem[Watanabe2009]{Watanabe2009}{\sc S. Watanabe}, Algebraic Geometry and Statistical Learning Theory, Cambridge University Press,
2009.
\bibitem[Witting1985]{Witting1985}{\sc H. Witting},  Mathematische Statistik I. Parametische Verfahren bei festem Stich-probenumfang. Teubner, Stuttgart, 1985.
Volume 18, Number 4,  2005, 779-822.
\end{thebibliography}
\end{document}